\documentclass[a4paper,12pt]{article}
\usepackage{amsmath,amsthm,amsfonts,amssymb,bbm}
\usepackage{graphicx,psfrag,subfigure,color,tikz,float}

\numberwithin{equation}{section}

\newcommand{\Pb}{\mathbb{P}}

\newcommand{\dx}{\mathrm{d}}
\newcommand{\R}{\mathbb{R}}
\newcommand{\N}{\mathbb{N}}
\newcommand{\Z}{\mathbb{Z}}

\newcommand{\GOE}{\mathrm{GOE}}
\newcommand{\GUE}{\mathrm{GUE}}

\usepackage[noadjust]{cite}

\footnotetext{\emph{Key words:} KPZ Universality, TASEP, Shocks, Last Passage Percolation.   \textit{Math. Subj. Class. :} 60K35}

\newtheorem{assumption}{Assumption}
\newtheorem{prop}{Proposition}[section]
\newtheorem{thm}[prop]{Theorem}
\newtheorem{lem}[prop]{Lemma}

\newtheorem{cor}[prop]{Corollary}

\newtheorem{cla}[prop]{Claim}

\newtheorem{rem}[prop]{Remark}

\title{Transition to Shocks  in TASEP and Decoupling of Last Passage Times}
\author{
Peter Nejjar\thanks{peter.nejjar@ist.ac.at.  IST Austria, 3400 Klosterneuburg, Austria. Partially supported by ERC Advanced Grant No. 338804 and ERC Starting Grant No. 716117.}
}

\date{\today}

\begin{document}
\maketitle
\sloppy

\vfill
\begin{abstract}
We consider the totally asymmetric simple exclusion process  in a critical scaling parametrized by $a\geq0$, which creates a shock in the particle density of order $aT^{-1/3},$ $T$ the observation time.  When at $a=0$ one has step initial data, we provide bounds on the limiting law of particle positions for $a>0,$ which in particular imply that  in the double limit $\lim_{a \to \infty}\lim_{T \to \infty}$ one recovers  the product  limit law  and  the degeneration of the correlation length observed at   shocks of order $1$. This result can be phrased in terms of a general last passage percolation   model. We also obtain bounds on the decoupling of two-point functions of several $\mathrm{Airy}$  processes.
\end{abstract}
\vfill
\section{Introduction}
We consider the totally asymmetric simple exclusion process (TASEP). In this model,  particles move on $\Z$ and jump one step to the right with rate 1, subject to the exclusion constraint that there is at most one particle on each site such that particles attempting to jump to an occupied site stay put. A particle configuration at time $T\geq 0$ can be encoded by an $\eta_{T} \in \{0,1\}^{\Z}$  ($\eta_{T}(i)=1$ if $i$ is occupied at time $T$, $\eta_{T}(i)=0$ if not) and $(\eta_{T})_{T\geq 0}$ is the TASEP, see \cite{Li85b} for  its  rigorous construction.

 Given an initial configuration $\eta_{0}$  in  TASEP, let us attach a label $n\in \Z$ to each particle  and denote by $x_{n}(T)$ the position of particle $n$ at time $T\geq 0$. 
 Depending on  $\eta_{0}$, one has different   large time  densities of particles $\rho(\xi)$, where, informally, $\rho(\xi)$ is the probability that there is a particle at $\lfloor \xi T \rfloor $ for $T$ large. Formally, $\rho$ is the density function of the measure to which the rescaled empirical particle density converges vaguely as $T\to \infty$ :
 
 \begin{equation}  \label{partdens}
\lim_{T \to \infty}\frac{1}{T}\sum_{i \in \Z}\delta_{\frac{i}{T}}\eta_{\,T}(i)=\rho(\xi)\dx \xi,
\end{equation}
 with $\delta_{x}$ the Dirac measure, and $\rho$ is  the unique entropy solution to the Burgers equation.
 Consider for instance the initial data 
  \begin{equation}\label{stepa}
 x_{n}(0)=\begin{cases}-n, \quad \quad \quad \quad \quad \, &\mathrm{for \,\,}-\lfloor a T^{2/3}\rfloor \leq n\leq 0
 \\ -n-\lfloor aT^{2/3}\rfloor  \quad  &\mathrm{for \,\,} n\geq 1,
 \end{cases}
 \end{equation}
 where $a\geq 0$ is a constant. The density profile  created by this initial data does not depend on $a$ and  is given in Figure \ref{densfig} left : It has a region where the  particle density is linearly decreasing, which  is called a \textit{rarefaction fan}.

 \begin{figure}[H]\label{densfig}
 \begin{center}
   \begin{tikzpicture}
     \draw [very thick,->] (-2,0) -- (2,0) node[below=4pt] {\small{$\xi$}};
  \draw [very thick,->] (0,-0.5) -- (0,2.5);
    \draw (0.3,1.3) node[anchor=south]{\small{$1$}};
    \draw[red,thick] (-2,1.3) -- (-1,1.3);
    \draw[red,thick] (-1,1.3)--(1,0);
    \draw[very thick] (-1,-0.1) -- (-1,0.1);
       \draw (-0.8,-0.6) node[anchor=south]{\small{$-1$}};
    \draw[very thick] (1,-0.1) -- (1,0.1);
       \draw (1.2,-0.6) node[anchor=south]{\small{$1$}};
\filldraw(0,1.3) circle(0.08cm); 
    \draw (-0.1,2) node[anchor=east] {\small{$\rho(\xi)$}};

\begin{scope}[xshift=7cm]
 \draw [very thick,->] (-2,0) -- (2,0) node[below=4pt] {\small{$\xi$}};
  \draw [very thick,->] (0,-0.5) -- (0,2.5);
    \draw (0.3,1.3) node[anchor=south]{\small{$1$}};
    \draw[red,thick] (-2,1.3) -- (-1.7,1.3);
    \draw[red,thick] (-1.7,1.3)--(0,0.45);
    \draw[very thick, red] (0,0.9) -- (1.8,0);

\filldraw(0,1.3) circle(0.08cm); 
    \draw (-0.1,2) node[anchor=east] {\small{$\rho(\xi)$}};

   \end{scope}
   \end{tikzpicture}    \end{center} \caption{Left: Large time density of TASEP started from initial data \eqref{stepa}. The density $\rho$ decreases linearly from $1$ to $0$ in the interval $[-1,1]$.
     Right: Density profile for TASEP  started from initial data \eqref{shock}. At the origin, two regions of decreasing density come together and the density $\rho$ jumps from $(1-\beta)/2$  to $(1+\beta)/2,$ i.e. there is a shock.}
\end{figure}
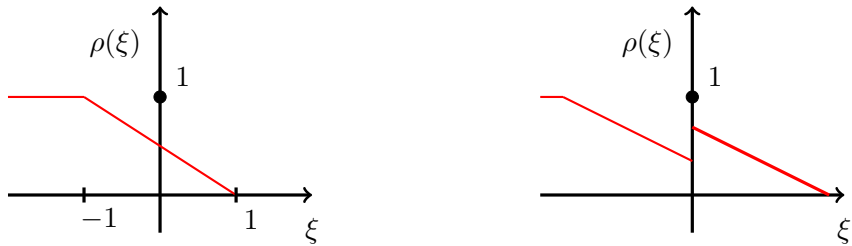

Consider now in contrast for   $\beta \in (0,1)$ the initial configuration
 \begin{equation}\label{shock}
 \tilde{x}_{n}(0)=\begin{cases}-n, \quad \quad \quad \quad \quad \, &\mathrm{for \,\,}-\lfloor \beta  T\rfloor \leq n\leq 0
 \\ -n-\lfloor \beta T\rfloor  \quad & \mathrm{for \,\,} n\geq 1.
 \end{cases}
 \end{equation}
The density profile has two distinct regions where the density is decreasing  and there is a discontinuity between them, see Figure   \ref{densfig} right. This is called a \textit{shock}, which in  this case is located at the origin.

The main topic of this paper is the transition of fluctuations of particle positions  when $\rho$ is continuous to when $\rho$ has  a shock. The fluctuations in these two situations are known to be very different (see  \eqref{step}, \eqref{shockb} below for the precise statements):

 If we choose $\nu>0$ so that $\tilde{x}_{\nu T}(T)$ is  located at the shock, i.e. at the origin,   the fluctuations of  $\tilde{x}_{\nu T}(T)$ are given  by a product of two  Tracy-Widom $F_{\mathrm{\GUE}}$ distributions and particles are non-trivially correlated on the $T^{1/3}$ length scale.  
 For $x_{\nu (a) T}(T)$ (where $\nu (a)$ is chosen so that  $x_{\nu (a) T}(T)$ is  located at the origin, see  \eqref{number}) and $a=0,$ however, the fluctuations are given  by a single Tracy-Widom $F_{\mathrm{\GUE}}$ distribution and particles are non-trivially correlated on the $T^{2/3}$ scale. 
 
If we - illegally -  set $a=\beta T^{1/3}$  (since $a,\beta $ are fixed constants independent of $T,$ one cannot have $a=\beta T^{1/3}$),  then $x_{n}$ and $\tilde{x}_{n}$ coincide and (trivially) have the same fluctuation behavior. Continuing with this informal heuristics,  letting $T \to \infty$ in $a=\beta T^{1/3}$ leads to $a\to \infty$  (for $\beta$ fixed)  as well as to $\beta \to 0$ (for $a$ fixed).  So if we want the two fluctuation behaviors to coincide, it seems reasonable to  consider the double limit $\lim_{a\to \infty}\lim_{T\to \infty}$ of $x_{\nu(a) T}(T),$ and the double limit $\lim_{\beta \to 0}\lim_{T\to \infty}$ of $\tilde{x}_{\nu T}(T).$   
As the following corollary  of our main result, Theorem \ref{thmtasep}, shows,  it is indeed with these two double limits that  a continuous transition between the two scaling regimes occurs:
\begin{cor}\label{corintro}Consider the initial data $x_{n},\tilde{x}_{n}$ from \eqref{stepa}, \eqref{shock} and let $u \in \R, \xi=\frac{u}{2}\frac{\beta-1}{\beta}.$   Then 
\begin{align}&\label{abc}\lim_{a\to \infty } \lim_{T \to \infty}\Pb \bigg( x_{\lfloor \frac{T}{4}-T^{2/3}\frac{a+\frac{u}{a}}{2}+T^{1/3}\frac{(\frac{u}{a}+a)^{2}}{4}\rfloor}(T)\geq T^{2/3}\frac{u}{a}-\frac{T^{1/3}}{2^{1/3}}s \bigg)
\\&=\label{abcd}\lim_{\beta \to 0} \lim_{T \to \infty}\Pb \bigg( \tilde{x}_{\lfloor        T\frac{(1-\beta)^{2}}{4}+\xi T^{1/3}         \rfloor}(T)\geq T^{1/3}\frac{u}{\beta} -\frac{T^{1/3}}{2^{1/3}}s \bigg)
\\&=\label{lastid}F_{\GUE}(s)F_{\GUE}(s-u2^{4/3}).
\end{align}
\end{cor}
Note that if  we again formally set $a=\beta T^{1/3}, $  and take  $u \in \R, \xi=\frac{u}{2}\frac{\beta-1}{\beta},$ then we have  
\begin{equation}\label{number}
 \frac{T}{4}-T^{2/3}\frac{a+\frac{u}{a}}{2}+T^{1/3}\frac{(\frac{u}{a}+a)^{2}}{4}=T\frac{(1-\beta)^{2}}{4}+\xi T^{1/3} +\mathcal{O}(T^{-1/3}),
\end{equation}
which motivates the choice of the particle number in \eqref{abc}.

 Finally,  the identity \eqref{lastid} is a simple consequence of the known convergence \eqref{shockb}.
See also after Theorem \ref{thmtasep} for a further discussion of this result.

\subsection{TASEP and Last Passage Percolation}
Here we introduce the notation for the two models considered in this paper, namely TASEP and Last Passage Percolation (LPP).
We consider TASEP with particles labelled from right to left, i.e., when $x_{n}(T)$ denotes the position of particle number $n \in \Z$ at time $T$ 
we have
\begin{equation}
\cdots<x_{2}(0)<x_{1}(0)<x_{0}(0)<x_{-1}(0)<x_{-2}(0)\cdots,
\end{equation}
note this order is preserved in time. TASEP is in one-to-one correspondence with last passage percolation, which we define next. Fix $(m,n)\in \Z^{2}$ (the end point) and $\mathcal{L}\subseteq \Z^{2}$ (the starting set). Let $\{\omega_{i,j}\}_{(i,j) \in \Z^{2}}$ be nonnegative random variables, seen as weights at the point $(i,j).$ An up-right path $\pi=(\pi(0),\ldots,\pi(k))$ from $\mathcal{L}$ to $(m,n)$  is a sequence of points with $\pi(0)\in \mathcal{L},\pi(k)=(m,n), \pi(i)-\pi(i-1)\in \{(0,1),(1,0)\}$. Then the LPP time from $\mathcal{L}$ to $(m,n)$ is defined as 
\begin{equation}\label{LPP}
L_{\mathcal{L}\to (m,n)}=\max_{\pi : \mathcal{L} \to (m,n)}\sum_{(i,j)\in \pi}\omega_{i,j}
\end{equation}
where the maximum in \eqref{LPP} is taken over all up-right paths from $\mathcal{L} $ to $(m,n)$. We will only consider $\{\omega_{i,j}\}_{(i,j)\in \Z^{2}}$ such that there is a.s.  a unique path $\pi$ where the maximum
\eqref{LPP} is attained, and we denote this path by $\pi^{\mathrm{max}}$. When there are no or infinitely many  paths from $\mathcal{L} $ to $(m,n)$, we set, say, $L_{\mathcal{L}\to (m,n)}= \infty,$ also \eqref{LPP} straightforwardly generalizes to several end points.  

Given an initial data $\{x_{n}(0)\}_{n\in I},I\subset \Z,$ of TASEP we set 
$\mathcal{L}=\{(x_{n}(0)+n,n),n \in I\}$. Assuming all particles have an exponential clock with parameter $1$ we take $\{\omega_{i,j}\}_{(i,j)\in \Z^{2}}$ independent, and  $\omega_{i,j}\sim \exp(1)$ if   $(i,j)\notin \mathcal{L}$ and $\omega_{i,j}=0$ for $(i,j)\in \mathcal{L}.$ 
With this choice,
the link between TASEP and LPP is given by 
\begin{equation}\label{LPPTASEP}
\Pb(x_{n}(T)\geq m-n)=\Pb(L_{\mathcal{L}\to (m,n)}\leq T).
\end{equation}

\subsection{Related Works}\label{relwork}
Here we rehash some of the known results that we will use in this paper, and describe the method of proof from \cite{FN14}, which is relevant to this paper.

A result we shall use repeatedly in this paper concerns the transversal fluctuations of the maximizer $\pi^{\mathrm{max}}:$ when  $\omega_{i,j}$  are i.i.d.  and  $\omega_{i,j} \sim \exp(1),$ there are bounds on the probability that the maximizer 
$\pi^{\mathrm{max}}$ from $(0,0)$ to a point  $( \tau T, T)$ deviates more than $k T^{2/3}$ from the straight line $\{(\kappa \tau T, \kappa T)$,  $0\leq \kappa \leq 1\}$. This is a result from \cite{BSS16} and   cited here as Theorem 3.2. 
Furthermore, for line-to-point problems, we have detailed control over the distribution of the (random) starting point of $\pi^{\max}$, see  
\cite{Pi17}, Lemma 1.1, Lemma 1.2 and also (4.18) in \cite{FO17}.

The fluctuation behavior at shocks, rarefaction fans and flat (constant density) TASEP have been obtained in detail.  We collect  here the relevant results.

\begin{itemize}
\item[i)] Rarefaction fan: By Theorem 1.6 of  \cite{Jo00}, for  the initial data \eqref{stepa}  for $a=0$ we have  that, for $u \in \R$ 
 \begin{equation}\label{step}
\lim_{T \to \infty}\Pb\left(\frac{x_{\lfloor T/4+u 2^{-2/3}T^{2/3}\rfloor }(T)+u 2^{1/3}T^{2/3}-u^{2}T^{1/3}2^{-1/3}}{-T^{1/3}2^{-1/3}}\leq s\right)=F_{\GUE}(s).
\end{equation}
 
\item[ii)]Flat TASEP: For flat TASEP, we have the following result, which is  a reformulation of Theorem 2.8 in \cite{FO17}. For  $\varrho\in (0,1),$  $x^{\varrho}_{n}(0)=-\lfloor n/\varrho\rfloor $ and $u \in \R$, we have that 
\begin{equation}\label{flatT}
\lim_{T\to \infty}\Pb\left(x_{\lfloor \varrho(1-\varrho)T+uT^{2/3}\rfloor }^{\varrho}(T)\geq -\frac{u}{\varrho}T^{2/3}-\frac{(1-\varrho)^{2/3}}{\varrho^{1/3}}T^{1/3}s\right)=F_{\GOE}(2^{2/3}s).
\end{equation}

\item[iii)] Shocks: In the shock case \eqref{shock}, we have  (see \cite{FN14}, Corollary 2.7) that 
\begin{equation}\begin{aligned}\label{shockb}
\lim_{T\to\infty}\Pb\left( \tilde{x}_{\lfloor \frac{(1-\beta)^2}{4} T+\xi T^{1/3}\rfloor } (T) \geq   -sT^{1/3} \right) =&F_{\GUE}\left(\frac{s-\xi/\rho_1}{\sigma_1}\right)\\&\times F_{\GUE}\left(\frac{s-\xi/\rho_2}{\sigma_2}\right)
\end{aligned}\end{equation}
 for $\xi \in \R,$ $\rho_1=\frac{1-\beta}{2}$, $\rho_2=\frac{1+\beta}{2}$, $\sigma_1=\frac{(1+\beta)^{2/3}}{2^{1/3}(1-\beta)^{1/3}}$, and $\sigma_2=\frac{(1-\beta)^{2/3}}{2^{1/3}(1+\beta)^{1/3}}$.
 \end{itemize}
  In fact, in \cite{FN14},  a general Theorem is presented which is then applied to other shock initial data as well. 
   This general Theorem is formulated in terms of LPP. By the link \eqref{LPPTASEP} 
 the distribution of the particle position at the shock (e.g. $\tilde{x}_{\lfloor \frac{(1-\beta)^2}{4} T+\xi T^{1/3}\rfloor } (T)$ from \eqref{shockb}) is equivalent to studying $\max\{L_{\mathcal{L}^{-}\to E},L_{\mathcal{L}^{+}\to E}\}$
 where $E\in \Z^{2}$ is chosen to be at the shock and $\mathcal{L}^{+}$ (resp. $\mathcal{L}^{-}$) lie in the upper left (resp. lower right) quadrant. The main observation of \cite{FN14} is that $L_{\mathcal{L}^{-}\to E},L_{\mathcal{L}^{+}\to E}$ decouple as $T\to \infty$. 
 
 This decoupling is based on three facts:
\begin{itemize}
\item[ i)] The maximizers  $\pi^{\mathrm{max}}_{\mathcal{L}^{+}\to E},$$\pi^{\mathrm{max}}_{\mathcal{L}^{-}\to E}$ start at points with distance $\mathcal{O}(T)$ 
 with probability $1$ as $T \to \infty.$ 
\item[ii)] The  transversal fluctuations of   $\pi^{\mathrm{max}}_{\mathcal{L}^{+}\to E},$$\pi^{\mathrm{max}}_{\mathcal{L}^{-}\to E}$ are $\mathcal{O}(T^{2/3}).$ [In fact, in \cite{FN14} it suffices to know they are $o(T^{\chi})$ for some $\chi<1$.]
 \item[iii)] The slow decorrelation phenomenon  \cite{CFP10b}:  Consider  a point $E^{+}$ on the characteristic line joining $\mathcal{L}^{+}$ and $ E
$ (the characteristic line is the deterministic line which $\pi^{\mathrm{max}}_{\mathcal{L}^{+}\to E}$, on the $\mathcal{O}(T)$  scale, follows). If  $|| E^{+}-E||=\mathcal{O}(T^{\nu}),\nu <1$, then  $T^{-1/3}(L_{\mathcal{L}^{+}\to E^{+}}+\mu T^{\nu} -L_{\mathcal{L}^{+}\to E})$ converges to zero in probability when choosing the right value for $\mu$.
\end{itemize}

Point iii) allows to replace $L_{\mathcal{L}^{+}\to E}$ by $L_{\mathcal{L}^{+}\to E^{+}}$ . If we take $\nu >2/3,$   then by points i), ii),   $\pi^{\mathrm{max}}_{\mathcal{L}^{+}\to E^{+}},$$\pi^{\mathrm{max}}_{\mathcal{L}^{-}\to E}$ stay in (deterministic) disjoint sets with very high probability, i.e. $L_{\mathcal{L}^{-}\to E},L_{\mathcal{L}^{+}\to E^{+}}$  are asymptotically independent, leading to the result.

Furthermore, in \cite{FN15}, we had studied the critical scaling for flat TASEP, where the shock is created by the presence of different speeds in the model, and numerically obtained the transition to the product structure of \cite{FN14}.
Finally, after this work was posted on arxiv, the transition to shock fluctuations for flat TASEP was obtained in the recent work \cite{QR18}, by completely different methods than ours, namely  new exact determinantal formulas for TASEP, and without using the LPP picture.

\subsection{Contributions of this paper}

This paper is the first to study the transition of fluctuations when the density is smooth to the fluctuations when there is a shock.
Corollary \ref{corintro}, and all other results in this paper, are obtained by working in the last passage percolation (LPP) picture.
  In terms of LPP, studying the transition to shock fluctuations means to study the maximum of two last passage times which remain correlated for all $T >0$, but which, as we show, decouple in a double limit $\lim_{a \to \infty} \lim_{T \to \infty},$ where $a$ is an extra parameter in the TASEP/LPP model.

  As a concrete model, corresponding to the initial data \eqref{stepa}, we consider in  Theorem \ref{Thmpp} for $a\geq 0$ the starting sets
 $\mathcal{L}^{+}=(-\lfloor a T^{2/3}\rfloor,0),\mathcal{L}^{-}=(0,-\lfloor a T^{2/3}\rfloor)$ and an end point $E=(\lfloor T+\frac{u}{a}T^{2/3}\rfloor,\lfloor T\rfloor).$  For $a>0$, we are in a critical scaling. A  lower bound for $ \Pb(\max\{L_{\mathcal{L}^{+}\to E},L_{\mathcal{L}^{-}\to E}\}\leq s)$ is provided by the FKG inequality,  so the main work is to find a suitable upper bound, which we do in    Theorems \ref{thmtasep} and \ref{Thmpp}. These bounds in particular imply 
  that  one recovers, in the double limit $\lim_{a \to \infty}\lim_{T \to \infty},$ the product structure  of \cite{FN14} as in Corollary \ref{corintro}.

 One can deduce from Theorem \ref{Thmpp}  a statement about the decoupling of the two-point function of the $\mathrm{Airy}_2$ process, see Corollary \ref{CorA2}. While more precise statements  than ours are available (see \cite{TrWid11},\cite{Wid03}, \cite{AvM03}), our proof is new and probabilistic as we make use of the convergence in LPP; which gives some intuition as to why the decoupling happens.

We also consider shocks which, unlike in \eqref{shockb}, are not between two regions of decreasing density, but two regions of (different) constant densities $\varrho_{1}>\varrho_{2}$. 
In this case, however, the fluctuations of the macroscopic shock, i.e. the analogue of \eqref{shockb}, has not been obtained\footnote{In Corollary 2.5  of \cite{FN14}, a shock between regions of constant density was considered, which is however created by slow particles.}, but see \cite{FGN17}, Section 2 for computations in this direction.  

We prove the fluctuations for such a macroscopic  shock in Theorem \ref{shockflat}, which gives a product of two $F_{\GOE}$ distributions.
The analogue of Corollary \ref{corintro}, i.e. the transition of the fluctuations of TASEP with constant density to the fluctuations of Theorem \ref{shockflat}, is stated in \eqref{you}, the proof is however only sketched  at the end of  Section \ref{othgen}. 

Theorem \ref{Thmpp} can be seen as an instance of a  general  Theorem about the decoupling of last passage times under some assumptions, see Theorem \ref{GenGen}. Theorem \ref{GenGen} is much simpler than the general  Theorem 2.1 of \cite{FN14}, and at the same time, gives a stronger result, as it provides some upper and lower bounds. Furthermore,  Theorem \ref{GenGen} gives the framework to show  the  decoupling of the $\mathrm{Airy}_1,\mathrm{Airy}_{2\to 1}$ processes (see \cite{BFS07},\cite{BF07} for definitions), which has not been done before, see Theorem \ref{airydec}.   The decoupling of these processes corresponds to the decoupling of last passage times $L_{\mathcal{L}\to E_1},L_{\mathcal{L}\to E_2}$ where $\mathcal{L}$ is now a (half-) line and the points $E_1,E_2$ have distance $a T^{2/3}$ from each other. Finally,  as the simplest example of decoupling, we show in Theorem \ref{twotime} the decoupling along the \textit{time-like} direction in exponential LPP. Recently, exact formulas (\cite{Jo17},\cite{Jo18}) for the two time distribution have been found, and such a  decoupling had been expected to occur,   see \cite{Jo17}, Remark 2.3.

 \subsubsection{Methods of proof}
    Let us briefly describe how our methods differ from related work, in particular \cite{FN14} (see Section  \ref{relwork}). First, and unlike in \cite{FN14}, the maximizing paths now start at distance $\mathcal{O}(T^{2/3})$ from each other, which is the scale of their transversal fluctuations. This requires us   to use a   refined control over transversal fluctuations of maximizers in LPP, using results from \cite{BSS16}, see Theorem \ref{trans1}. The second main  ingredient to Theorem \ref{Thmpp} is an extended slow decorrelation result.  Namely, since we have two maximizers which start in points with distance $\mathcal{O}(T^{2/3})$ and go to  $E,$ the maximizers  will come together   already at distance $\mathcal{O}(T)$ from $E$. Consequently, in order to obtain independent passage times,  we wish to replace      $L_{\mathcal{L}^{+}\to E}$ by $L_{\mathcal{L}^{+}\to E^+},$ with $E^+$  on the straight (characteristic) line from $\mathcal{L}^{+} $ to $E$ and at distance $\varepsilon T$ from $E$. If $\varepsilon$ is not too small, the probability that the maximizers of $L_{\mathcal{L}^{+}\to E^+},L_{\mathcal{L}^{-}\to E}$ cross  will vanish for $a,T $ large. In the usual slow decorrelation (see Theorem 2.1 in \cite{CFP10b}), used in \cite{FN14},   $E^{+}$ is at distance 
      $T^{\nu},\nu <1,$ from $E$  such that  the fluctuations from $E^{+}$ to $E$ vanish under the $T^{1/3}$ scaling. In our situation, however, they do not vanish  as $T \to \infty.$ Nevertheless, they  are only of order $\varepsilon^{1/3}T^{1/3}$; in particular, they vanish in the double limit $\lim_{\varepsilon \to 0}\lim_{T \to \infty}$. 
      We show that it is possible to choose $\varepsilon= \varepsilon(a)$ in such a way that $\varepsilon(a)$ goes to zero with $a$, but is large enough so that the maximizers stay in disjoint sets with high probability (see Section \ref{proofsec}), leading to Theorem \ref{Thmpp}.
      
To show the decoupling of the $\mathrm{Airy}_1,\mathrm{Airy}_{2\to 1}$ processes, no slow decorrelation result is needed, but a control over  transversal fluctuations and  the (random) starting point of the maximizing path, the latter  was obtained recently in \cite{Pi17} and \cite{FO17}. 
 Finally, for the macroscopic shock proven in Theorem \ref{shockflat}, no extended slow decorrelation or refined control over transversal fluctuations and starting points is needed, but we do need the universality of the $F_{\GOE}$ distribution in flat TASEP, cited here in \eqref{flatT}.

 \textit{Outline.} In Section \ref{modres} we state all the results obtained in this paper. They concern either point-to-point LPP problems, or line-to-point LPP problems.
 
    In Section \ref{proofsec} we prove all results that concern point-to-point problems, and the general Theorem \ref{GenGen}. Specifically,  Corollary  \ref{corintro}, Theorem \ref{thmtasep}, Theorem \ref{Thmpp}, Corollary \ref{CorA2},  Theorem \ref{twotime}
    and Theorem  \ref{GenGen} are proven in Section \ref{proofsec}.
 In Section \ref{othgen} we prove all results that concern line-to-point problems: These are Theorem \ref{shockflat} and Theorem \ref{airydec}. Finally, an outline of the proof for the transition to shock fluctuations for flat TASEP, stated in \eqref{you}, is given.
 
\textbf{Acknowledgements} We thank M\'{a}rton Bal\'{a}zs for discussing \cite{BCS06} with us, and  Patrik Ferrari and Zhipeng Liu for useful discussions regarding this paper, as well as the anonymous referee for helpful comments.

\textbf{Notation} We denote for $x \in \R$  by $\lfloor x \rfloor $ the largest $z \in \Z$ with $z \leq x$,  and $T,t$ are always large time parameters which go to infinity. $C,c$ denote constants whose exact values are  immaterial and do not depend on the parameters present (mostly $a,t,k$).

\section{Main Results}\label{modres}

\subsection{Transition to Shock Fluctuations}
 The following Theorem provides the transition from the fluctuations of TASEP with step initial data to the fluctuations at the $\GUE-\GUE$ shock of \eqref{shockb}.

\begin{thm}\label{thmtasep}
Let $x_{n}(0)=-n$ for $-\lfloor a T^{2/3}\rfloor \leq n\leq 0$ and $x_{n}(0)=-n-\lfloor aT^{2/3}\rfloor $ for $n \geq 1$. 
Then there are constants $C, c>0$ such that for any $0<k<a, \delta >0$ and $\frac{k}{a}<\varepsilon(a)<1$ we may bound \begin{align*}
&F_{\GUE}(s)F_{\GUE}(s-u2^{4/3})\\&\leq \lim_{T \to \infty}\Pb \bigg( x_{\lfloor \frac{T}{4}-T^{2/3}\frac{a+\frac{u}{a}}{2}\rfloor}(T)\geq \frac{u}{a}T^{2/3}+T^{1/3}\frac{(\frac{u}{a}+a)^{2}}{2}-\frac{T^{1/3}}{2^{1/3}}s \bigg)
\\&\leq F_{\GUE}\left(\frac{s+\delta}{(1-\varepsilon(a))^{1/3}}\right)F_{\GUE}\left(s-u2^{4/3}\right)\\&+F_{\GUE}(-\delta \varepsilon(a)^{-1/3})+Ce^{-ck}.
\end{align*}

\end{thm}
 In particular, Theorem \ref{thmtasep} directly  implies  the following:
\begin{equation}
\begin{aligned}\label{astep}
\lim_{ a \to \infty}\lim_{T \to \infty}\Pb \bigg( x_{\lfloor \frac{T}{4}-T^{2/3}\frac{a+\frac{u}{a}}{2}\rfloor}(T)\geq \frac{u}{a}T^{2/3}+&T^{1/3}\frac{(\frac{u}{a}+a)^{2}}{2}-\frac{T^{1/3}}{2^{1/3}}s \bigg)\\&=F_{\GUE}(s)F_{\GUE}(s-u2^{4/3}).
\end{aligned}
\end{equation}

By   taking $u=a \tilde{u}$ in \eqref{astep} such that $ u/a=\tilde{u}$ and then setting $a=0$, one has the usual step initial data  and  the $T \to \infty$ limit in \eqref{astep} gives the $\mathrm{Airy}_2$ process $\mathcal{A}_{2}(\tilde{u})_{\tilde{u} \in \R}$ .  
To recover the shock situation, one should transfer the $T^{1/3}\frac{(\frac{u}{a}+a)^{2}}{2}$ term in the particle number, i.e. consider
\begin{equation}\label{heuri}
\Pb \bigg( x_{\lfloor \frac{T}{4}-T^{2/3}\frac{a+\frac{u}{a}}{2}+T^{1/3}\frac{(\frac{u}{a}+a)^{2}}{4}\rfloor}(T)\geq \frac{u}{a}T^{2/3}-\frac{T^{1/3}}{2^{1/3}}s \bigg).
\end{equation}
To create a macroscopic shock, set, for $\beta \in (0,1),$  $a=\beta T^{1/3},\xi=\frac{u}{2}\frac{\beta-1}{\beta}$, so that \eqref{heuri}
becomes (recall $\tilde{x}_{n}$ from \eqref{shock})
\begin{equation}
\Pb \bigg( \tilde{x}_{\lfloor        T\frac{(1-\beta)^{2}}{4}+\tilde{\xi}T^{1/3}         \rfloor}(T)\geq T^{1/3}u/\beta -\frac{T^{1/3}}{2^{1/3}}s \bigg).
\end{equation} 
Now \eqref{shockb}  implies
\begin{equation}\label{tada}
\lim_{\beta \to 0} \lim_{T \to \infty}\Pb \bigg( \tilde{x}_{\lfloor        T\frac{(1-\beta)^{2}}{4}+\tilde{\xi}T^{1/3}         \rfloor}(T)\geq T^{1/3}u/\beta -\frac{T^{1/3}}{2^{1/3}}s \bigg)=F_{\GUE}(s)F_{\GUE}(s-u2^{4/3}),
\end{equation}
and Corollary \ref{corintro} follows from \eqref{astep} and \eqref{tada}.

Next we state our result for the fluctuations of  \textit{macroscopic} shocks created by two regions of constant density, where the decoupling already happens in the $T \to \infty$ limit.
\begin{thm}\label{shockflat}Consider TASEP with initial data given by $1>\varrho_{1}>\varrho_{2}>0$ and
\begin{equation}
x^{\varrho_1 , \varrho_{2}}_{n}(0)=\begin{cases}
-\lfloor n/\varrho_{1}\rfloor \quad &\mathrm{for} \quad n\leq 0 \\ -\lfloor n/\varrho_{2}\rfloor\quad &\mathrm{for} \quad n>0.
\end{cases}
\end{equation} Then we have with $c_{i}=(1-\varrho_{i})^{-2/3}\varrho_{i}^{1/3},i=1,2$ 
\begin{align}
\lim_{T \to \infty}&\Pb\left(x_{\lfloor \varrho_{1}\varrho_{2}T+\xi T^{1/3}\rfloor}^{\varrho_{1},\varrho_{2}}(T)\geq (1-\varrho_{1}-\varrho_{2})T-sT^{1/3}\right)\\&\label{c2s2}=F_{\GOE}(2^{2/3}(s-\xi/\varrho_{1})c_1)F_{\GOE}(2^{2/3}(s-\xi/\varrho_{2})c_2).
\end{align}
\end{thm}

The following is imminent from Theorem \ref{shockflat}.
\begin{cor}
We have that with $\tilde{\xi}=\xi(\varrho_{1}-\varrho_{2})^{-1}$ 
\begin{align}
\lim_{ \varrho_{1}\searrow \varrho_{2}}\lim_{T \to \infty}&\Pb\left(x_{\lfloor \varrho_{1}\varrho_{2}T+\tilde{\xi} T^{1/3}\rfloor}^{\varrho_{1},\varrho_{2}}(T)\geq (1-\varrho_{1}-\varrho_{2})T-(s/c_1+\tilde{\xi}/\varrho_{1})T^{1/3}\right)\\&=F_{\GOE}(2^{2/3}s )F_{\GOE}(2^{2/3}s-2^{2/3}\xi (1-\varrho_{2})^{-2/3}\varrho_{2}^{-5/3}).
\end{align}
\end{cor}
Now, to be in a critical scaling,  set  $\varrho_{1}(a)=\varrho_{2}+a T^{-1/3}$ and recall $\tilde{\xi}=\xi(\varrho_{1}-\varrho_{2})^{-1}.$ Then the analogue of Corollary \ref{corintro} is 
\begin{equation}
\begin{aligned}\label{you}
&\lim_{a \to + \infty}\lim_{T \to \infty}\Pb(x_{\lfloor \varrho_{1}(a) \varrho_{2}T+\xi T^{2/3}/a\rfloor}^{\varrho_{1}(a),\varrho_{2}}(T)\geq (1-\varrho_{1}(a)-\varrho_{2})T-\xi T^{2/3}/a\varrho_{2}-s T^{1/3}/c_{2})
\\&=\lim_{ \varrho_{1}\searrow \varrho_{2}}\lim_{T\to \infty}\Pb\left(x_{\lfloor \varrho_{1}\varrho_{2}T+\tilde{\xi} T^{1/3}\rfloor}^{\varrho_{1},\varrho_{2}}(T)\geq (1-\varrho_{1}-\varrho_{2})T-(s/c_1+\tilde{\xi}/\varrho_{1})T^{1/3}\right).\end{aligned}
\end{equation}

We do not provide a full proof of \eqref{you}, but see  the end of Section \ref{othgen}  for a discussion  and outline of proof of \eqref{you}.

\subsection{Decoupling of Last Passage Times}

In the following, we give our results which correspond to several  last passage times decoupling in a certain  double limit.
The first, Theorem \ref{Thmpp}, is the LPP counter part of  Theorem \ref{thmtasep}. By using that various $\mathrm{Airy}$ processes arise as limit in LPP models,  we show how decoupling of last passage times implies decoupling bounds for the $\mathrm{Airy}_1,\mathrm{Airy}_2 $ and $\mathrm{Airy}_{2\to 1}$ processes. The decoupling of the $\mathrm{Airy}_2$ process, reported in Corollary \ref{CorA2}, is a corollary of Theorem \ref{Thmpp}, whereas  the decoupling of the $\mathrm{Airy}_{2\to 1},\mathrm{Airy}_1$  processes require new proofs. Finally, all these decouplings of last passage times fall in the framework of a simple more general statement about decoupling of last passage times, see Theorem \ref{GenGen}. This improves Theorem 2.1 of \cite{FN15}.

\begin{thm}\label{Thmpp}
Set  $\mathcal{L}^{+}=(-\lfloor a t^{2/3}\rfloor ,0), \mathcal{L}^{-}=(0,-\lfloor a t^{2/3}\rfloor ), \mathcal{L}=\mathcal{L}^{+}\cup\mathcal{L}^{-}$  and define 
\begin{equation} 
\begin{aligned}
&\mu^{a}t=4t + 2 t^{2/3}(a+u/a)-\left(a+\frac{u}{a}\right)^{2}t^{1/3}/4.
\end{aligned}
\end{equation}
There are constants $C,c>0$ such that for  $a>k>0$, any $\delta>0$ and $k/a<\varepsilon(a)<1$ we may bound
 \begin{align*}
F_{\GUE}(s)F_{\GUE}\left(s-\frac{u}{2^{4/3}}\right)&\leq  \lim_{t \to \infty}\Pb\left(\frac{L_{\mathcal{L} \to (\lfloor t+\frac{u}{a}t^{2/3}\rfloor,\lfloor t\rfloor )}-\mu^{a}t}{2^{4/3}t^{1/3}}\leq s\right)
\\&\leq F_{\GUE}\left(\frac{s+\delta}{(1-\varepsilon(a))^{1/3}}\right)F_{\GUE}\left(s-\frac{u}{2^{4/3}}\right)\\&+F_{\GUE}(-\delta \varepsilon(a)^{-1/3})+Ce^{-ck}.
\end{align*}

 \end{thm}
\subsection{Decoupling of $\mathrm{Airy} $ Processes}
Theorem \ref{Thmpp} gives some estimates on the decay of the two point function of the $\mathrm{Airy}_2$ process $ \mathcal{A}_{2}$.  The two point function $\Pb(\mathcal{A}_{2}(0)\leq s_1, \mathcal{A}_{2} (a)\leq s_2)$ has already  been studied in detail (see in particular (7) in \cite{TrWid11}, and also the previous works  \cite{Wid03}, \cite{AvM03}). In particular, it is known that $\Pb(\mathcal{A}_{2}(0)\leq s_1, \mathcal{A}_{2} (a)\leq s_2)=F_{\GUE}(s_1)F_{\GUE}\left(s_2\right)+\mathcal{O}(a^{-2})$ as $a \to \infty$. However, the works \cite{TrWid11},\cite{Wid03}, \cite{AvM03} are all  based on Fredholm determinant (in \cite{Wid03}, \cite{TrWid11}) or PDE expression (in \cite{AvM03})  for the two point function, whereas we use that the   $\mathrm{Airy}_2$ process arises as limit in LPP.
\begin{cor}\label{CorA2} Let $a>k>0.$ 
Then for any $\delta>0$ and $1>\varepsilon(a)>k/a$ we may bound
 \begin{align*}
F_{\GUE}(s)F_{\GUE}\left(s-4u\right)&\leq \Pb\bigg(\mathcal{A}_{2}\left(-a-\frac{u}{a}\right)\leq s,\mathcal{A}_{2}\left(a-\frac{u}{a}\right)\leq s-4u\bigg)\\&\leq F_{\GUE}\left(\frac{s+\delta}{(1-\varepsilon(a))^{1/3}}\right)F_{\GUE}\left(s-4u\right)\\&+F_{\GUE}(-\delta \varepsilon(a)^{-1/3})+Ce^{-ck}.
\end{align*}

\end{cor}

In Section \ref{othgen}, we also study  the  decay of the joint distribution of the $\mathrm{Airy}_1,\mathrm{Airy}_{2 \to 1}$ processes, which we denote by $\mathcal{A}_1, \mathcal{A}_{2 \to 1}$. This  decoupling does not correspond to a transition to shock fluctuations, rather   one has  two maximizers which start and end in points with distance $a t^{2/3}.$ The starting point is random, and controlling it is an  extra ingredient required here, which was obtained in recently in \cite{FO17},\cite{Pi17}. The result we obtain is as follows.

 \begin{thm}\label{airydec} There are  $C,c,a_0>0$ such that for  $a > a_0,b\in \R$  we have 
 \begin{equation}
\begin{aligned}\label{decouplage}
F_{\GOE}(2s_1)F_{\GOE}(2s_2)&\leq\Pb( \mathcal{A}_1 (0)\leq s_1, \mathcal{A}_1 (a)\leq s_2)\\&\leq F_{\GOE}(2s_1)F_{\GOE}(2s_2) + Ce^{-ca}
\end{aligned} 
\end{equation}and  that for the $\mathrm{Airy}_{2\to 1}$  process we may bound 
\begin{align}
&\Pb(\mathcal{A}_{2\to 1}(b)\leq s_1)  \Pb(\mathcal{A}_{2\to 1}(|b|+a)\leq s_2)   \\&   \leq 
\Pb( \mathcal{A}_{2\to1} (b)\leq s_1, \mathcal{A}_{2\to1} (|b|+a)\leq s_2)
\\&\leq \Pb(\mathcal{A}_{2\to 1}(b)\leq s_1)  \Pb(\mathcal{A}_{2\to 1}(|b|+a)\leq s_2)         
+ Ce^{-ca}.
\label{decouplage2}
\end{align}

\end{thm}

\subsection{ Decoupling in the time-like direction}
Finally, as the simplest example of this paper, we show the decoupling of last passage percolation times along the time-like direction. Denote   for $x, y \in \R$ the points 
$P(x,y)=(\lfloor -  y (xt)^{2/3}\rfloor ,0)$
 and   $\mu(x,y)t=4xt-2y(xt)^{2/3}+\frac{y^{2}}{4}(xt)^{1/3}$ and denote
\begin{equation}
L_{P(x,y)\to (\lfloor xt\rfloor ,\lfloor xt\rfloor )}^{\mathrm{resc}}=\frac{ L_{P(x,y)\to (\lfloor xt\rfloor ,\lfloor xt\rfloor )}-\mu(x,y)t}{2^{4/3}(xt)^{1/3}}
\end{equation}
  For e.g.  points lying on a line with slope $1$, the decoupling we consider  corresponds to look for $\tau<a$ 
at 
\begin{equation}\label{twotimeex}
\lim_{t \to \infty}\Pb\left(\{L^{\mathrm{resc}}_{0\to (\lfloor \tau t\rfloor ,\lfloor \tau t\rfloor )}\leq s\}\cap  \{L^{\mathrm{resc}}_{0\to (\lfloor at\rfloor ,\lfloor at\rfloor )}\leq \zeta\}\right)
\end{equation}
and then let $a$ go to infinity. It is a priori not clear if \eqref{twotimeex} exists, hence we work with 
an arbitrary subsequential limit in \eqref{subseql}.
For the case of brownian  and geometric percolation, Johansson  proved in  \cite{Jo17}, \cite{Jo18} an explicit formula for \eqref{twotimeex}, see also \cite{BL18}  for results in  periodic TASEP, and \cite{FS16} for the decay of the covariance in the time-like direction. In  \cite{Jo17}, the author expects (see Remark 2.3 in \cite{Jo17}) that (the analogue of) \eqref{twotimeex} converges to $F_{\GUE}(s)F_{\GUE}(\zeta)$ as $a \to \infty$ and  notes that this can be checked heuristically but that it appears rather subtle. Here we show that a soft probabilistic argument suffices to show this decoupling, which extends   to the  multipoint two time distribution in exponential LPP, and even to provide some (non-optimal) bounds on the speed of decoupling. Note that the following  Theorem implies in particular that
\begin{align*}
&\lim_{a \to \infty}\lim_{ t_j \to \infty}\Pb\left( \bigcap_{i=1}^{l}\{L^{\mathrm{resc}}_{P(\tau,r_i)\to (\lfloor \tau t_j\rfloor ,\lfloor \tau t_j\rfloor )}\leq s_i\}\cap  \bigcap_{i=1}^{k}\{L^{\mathrm{resc}}_{P(a,u_i)\to (\lfloor at_j\rfloor ,\lfloor at_j \rfloor )}\leq \zeta_i\}\right)\\&=\Pb\left(\bigcap_{i=1}^{l}\mathcal{A}_{2}(r_i)\leq s_i\right)\Pb\left(\bigcap_{i=1}^{k}\mathcal{A}_{2}(u_i)\leq \zeta_i\right).
\end{align*}

\begin{thm}\label{twotime}
Let $a>\tau>0,$ and let the  $\{\omega_{i,j},i,j \in \Z\}$ be i.i.d. $\exp(1)$ distributed. Let $r_1 < \cdots < r_l$ and $u_1 < \cdots <u_k.$ Denote by $\lim_{ t_j \to \infty}$ an arbitrary subsequential limit. Then for any $\delta >0$ 
\begin{align}\label{2time}
&\Pb\left(\bigcap_{i=1}^{l}\mathcal{A}_{2}(r_i)\leq s_i\right)\Pb\left(\bigcap_{i=1}^{k}\mathcal{A}_{2}(u_i)\leq \zeta_i\right)
\\&\leq \lim_{ t_j \to \infty}\Pb\left( \bigcap_{i=1}^{l}\{L^{\mathrm{resc}}_{P(\tau,r_i)\to (\lfloor \tau t_j\rfloor ,\lfloor \tau t_j\rfloor )}\leq s_i\}\cap  \bigcap_{i=1}^{k}\{L^{\mathrm{resc}}_{P(a,u_i)\to (\lfloor at_j\rfloor ,\lfloor at_j \rfloor )}\leq \zeta_i\}\right)\label{subseql}
\\&\leq
 \Pb\left(\bigcap_{i=1}^{l}\mathcal{A}_{2}(r_i)\leq s_i\right)\Pb\left(\bigcap_{i=1}^{k}\mathcal{A}_{2}(u_i (1-\tau/a)^{1/3})\leq (\zeta_i+\delta)\frac{a^{1/3}}{(a-\tau)^{1/3}} \right)\\&+k F_{\GUE}(-\delta a^{1/3}\tau^{-1/3}).
\end{align}
\end{thm}

\subsection{General Theorem}
The preceding results can all be phrased in a simple Theorem about a general LPP model, which improves the general framework given in Theorem 2.1 of \cite{FN14}.
Let $\mathcal{L}^{+}, \mathcal{L}^{-}\subseteq \Z^{2}$  and let $\{\omega_{i,j},i,j \in \Z\}$ be independent exponentially distributed  weights.  We make three assumptions on our model.

\begin{assumption}\label{Assumpt1}
Let $t,a>0$ and assume there are $E_1 =E_1 (t,a), E_2 =E_2 (t,a) \in \Z^{2}$ 
and $\mu_{1}^{a},\mu_{2}^{a}>0$ such that
\begin{align}\label{eq4a}
& \lim_{t \to \infty}\Pb\left(\frac{L_{\mathcal{L}^{+}\to E_1}-\mu_1^{a} t}{t^{1/3}}\leq s\right)=G_{1}^{a}(s)
\\& \label{eq4b} \lim_{t \to \infty}\Pb\left(\frac{L_{\mathcal{L}^{-}\to E_2}-\mu_2^{a} t}{t^{1/3}}\leq s\right)=G_{2}^{a}(s),
\end{align}
where $G_{1}^{a}(s), G_{2}^{a}(s)$ are some distribution functions.
\end{assumption}
In Theorem \ref{Thmpp}, $G_{1}^{a}, G_{2}^{a}$ will be (shifted) $F_\GUE$ distributions, in Theorem \ref{airydec}, they will be $F_\GOE$ distributions.

\begin{assumption}\label{Assumpt2}
Assume there is a point $E^{+}=E_1 -(\kappa \varepsilon(a)t+d t^{2/3},\varepsilon(a)t)$ with $\kappa,\varepsilon(a)\geq 0,d \in \R$ such that for a $\mu^{\varepsilon(a)}\geq 0$ we have  
\begin{align}
&\lim_{t \to \infty}\Pb\left(\frac{L_{E^{+}\to E_1}-\mu^{\varepsilon(a)}t}{t^{1/3}}\leq s \right)=G_{0}^{a}(s)
\\&\lim_{t \to \infty}\Pb\left(\frac{L_{\mathcal{L}^{+}\to E^{+}}+\mu^{\varepsilon(a)}t-\mu_{1}^{a}t }{t^{1/3}}\leq s\right)=G_{1}^{a}(c_{\varepsilon(a)}s)
\end{align}
where $G_{0}^{a}$ is a distribution function, $c_{\varepsilon(a)}$ is a constant  and $G_{1}^{a}$ is from Assumption \ref{Assumpt1}.
\end{assumption}
In the context of Theorem \ref{Thmpp}, we will take $\varepsilon(a)>0,\lim_{ a \to \infty}\varepsilon(a)=0.$ Then, with $E$ as in Theorem \ref{Thmpp}, $\frac{L_{E^{+}\to E}-\mu^{\varepsilon(a)}t}{t^{1/3}}$ will vanish in the double limit $\lim_{a \to \infty}\lim_{t \to \infty}.$ \begin{assumption}\label{Assumpt3}
Assume there are independent random variables $\tilde{L}_{\mathcal{L}^{+}\to E^{+}},\tilde{L}_{\mathcal{L}^{-}\to E_2}$ such that for some $\tilde{\psi} \geq 0$ 
\begin{align}
\limsup_{t \to \infty}\Pb\left(\{\tilde{L}_{\mathcal{L}^{+}\to E^{+}}\neq L_{\mathcal{L}^{+}\to E^{+}}\} \cup\{ \tilde{L}_{\mathcal{L}^{-}\to E_2}\neq L_{\mathcal{L}^{-}\to E_2}\}\right)\leq \tilde{\psi}.
\end{align}
\end{assumption}
In Theorem \ref{Thmpp}, $\tilde{L}_{\mathcal{L}^{+}\to E^{+}},\tilde{L}_{\mathcal{L}^{-}\to E}$ will be last passage times with restricted transversal fluctuations,
in Theorem \ref{airydec} they will additionally have restricted starting points.

We denote
by 
\begin{equation}
L_{\mathcal{L}^{+}\to E_1}^{\mathrm{resc}}=\frac{L_{\mathcal{L}^{+}\to E_1}-\mu_{1}^{a}t}{t^{1/3}}
\end{equation}
and similarly denote by $L_{\mathcal{L}^{+}\to E_1}^{\mathrm{resc}},L_{E^{+}\to E_1}^{\mathrm{resc}} $ the LPP times rescaled as in Assumptions \ref{Assumpt1},\ref{Assumpt2}.
\begin{thm}\label{GenGen}
Under Assumptions \ref{Assumpt1},\ref{Assumpt2},\ref{Assumpt3} we have for any $\delta \geq 0$ 
\begin{align*}
G_{1}^{a}(s_1)G_{2}^{a}(s_2)&\leq \lim_{t_k \to \infty}\Pb(L_{\mathcal{L}^{+}\to E_1}^{\mathrm{resc}}\leq s_1, L_{\mathcal{L}^{-}\to E_2}^{\mathrm{resc}}\leq s_2)
\\& \leq G_{1}^{a}((s_1+\delta)c_{\varepsilon(a)})G_{2}^{a}(s_2)+G_{0}^{a}(-\delta)+3 \tilde{\psi},
\end{align*}
where $\lim_{t_k \to \infty}$ is any subsequential limit.
\end{thm}
Clearly, a version of Theorem \ref{GenGen} without taking the $t_k \to \infty$ limit also holds.
This could be used to refine the results of \cite{FN14} by obtaining upper and lower bounds  for finite $t$ in Theorem 2.1 in \cite{FN14} and its applications, instead of showing only the  convergence to a product as $t \to \infty$.

\section{Proof for point-to-point problems and Theorem \ref{GenGen}}\label{proofsec}

In this section, we prove the results which involve point(s)-to-point LPP problems, as well as the general Theorem \ref{GenGen} .
Corollary \ref{corintro}, Theorems  \ref{thmtasep},  \ref{Thmpp} and Corollary \ref{CorA2} are proved in Section \ref{mainmain}, Theorem \ref{twotime} and 
 Theorem \ref{GenGen} are proved in Section \ref{222}.

\subsection{Proof of Theorems \ref{thmtasep}, \ref{Thmpp} and Corollaries \ref{corintro}, \ref{CorA2}}\label{mainmain}
The proof of Corollary \ref{corintro} is immediate.
\begin{proof}[Proof of Corollary \ref{corintro}]
It follows from \eqref{astep} and  \eqref{tada}.
\end{proof}
Let us recall   the following result for point-to-point LPP.
\begin{prop}[Theorem 1.6 of~\cite{Jo00b}, Theorem 2 of \cite{BP07}]\label{JohConvergence}
Let $0<\eta<\infty, \eta=\eta_{0}+c\ell^{-1/3}$. Then,
\begin{equation} \begin{aligned}
\lim_{\ell\to\infty}\Pb\left(L_{0\to (\lfloor\eta\ell\rfloor,\lfloor\ell\rfloor)}\leq \mu_{\rm pp}\ell +s \sigma_\eta \ell^{1/3}\right)= F_\GUE(s)
\end{aligned}\end{equation}
where $\mu_{\rm pp}=(1+\sqrt{\eta})^2$, and $\sigma_\eta=\eta^{-1/6}(1+\sqrt{\eta})^{4/3}$. In particular, with $\mathcal{L}^{+}, \mathcal{L}^{-},\mu^{a}t$ as in Theorem \ref{Thmpp}, we have
\begin{equation}
\begin{aligned}
&\lim_{t \to \infty}\Pb\left(\frac{L_{\mathcal{L}^{+}\to (t+ u t^{2/3}/a,t)}-\mu ^{a}t}{2^{4/3}t^{1/3}}\leq s\right)=F_{\GUE}(s)
\\&\lim_{t \to \infty}\Pb\left(\frac{L_{\mathcal{L}^{-}\to (t+ u t^{2/3}/a,t)}-\mu ^{a}t}{2^{4/3}t^{1/3}}\leq s\right)=F_{\GUE}(s-u/2^{4/3}).
\end{aligned}
\end{equation}
\end{prop}

To proceed, we need bounds on the transversal fluctuations of maximizers in LPP.
 Let $(m, n)\in \Z^{2}_{\geq 0}$. Denote for $l\leq n$\  \begin{equation}Z_{l}(m,n)=Z_{l}=\max\{i:(i,l)\in \pi^{\mathrm{max}}_{0 \to (m,n)}\}\end{equation}where, with $A,B \in \Z^{2}_{\geq 0},$  $\pi^{\mathrm{max}}_{A \to B}$ is the maximizing path from $A$ to $B$ in the LPP model with independent  weights given  by $\omega_{i,j} \sim \exp(1).$     Similarly, define $Y_{r}^{\mathrm{TOP}}(m,n)$ to be the top-most point of $\pi^{\mathrm{max}}_{0 \to (m,n)}$ on the vertical line $i=r$. 
The following result was formulated for Poisson LPP (and for lines with bounded slope), but extends to the exponential model straightforwardly, see Section 13 in \cite{BSS16}.
\begin{thm}[Corollary 11.7 in \cite{BSS16}]\label{trans1}
Let $ d=d(t)=d_0 + \mathcal{O}(t^{-1/3}), d_0 \in \R$. Set  $m=\lfloor \eta_0 t +d t^{2/3}\rfloor,n=\lfloor t\rfloor. $ 
 There are  constants $C, c, >0$ such that for all $k>0$ 
\begin{align}\label{Z1}
&\limsup_{t\to \infty}\Pb(\max_{\kappa\in [0,1]}\{Z_{\lfloor \kappa t\rfloor}(m,n)- \kappa ( \eta_0 t +d t^{2/3})\} \geq kt^{2/3})\leq C e^{-ck}
\\&\label{Y1}\limsup_{t\to \infty}\Pb(\max_{\kappa\in [0,1]}\{Y_{\lfloor \kappa (\eta_0 t + d t^{2/3})\rfloor}^{\mathrm{TOP}}(m,n)- \kappa t \}\geq kt^{2/3})\leq  C e^{-ck}. \end{align}
\end{thm}

\begin{figure}[h]
  \centering
\begin{tikzpicture}
\draw (-0.5,-1) node[below=7pt] {$\mathcal{L}^{-}$};
\draw (-1.4,0.1) node[below=3pt] {$\mathcal{L}^{+}$};
\draw [ very thick, ->] (0,-3) -- (0,3) node[anchor=east] {$\mathbb{Z}$};
\filldraw (2.9,2.5)  circle (0.1 cm)   ;
\draw (2.9,2.5) node[anchor=west]{$E$};
\draw (0.2,0.68) node[anchor=west]{$R_{+}$};
\draw (1.1,0.7) node[anchor=west]{$R_{-}$};
\draw[dashed,very thick ] (-1,0) -- (2,1.9);
\draw[dotted,very thick] (0,-0.8) -- (2.,1.9);
\filldraw     (-1.4,0)      circle (0.1 cm);
\filldraw     (0,-1.2)      circle (0.1 cm);
\filldraw     (0.9,1.6)      circle (0.1 cm);
\draw  (0.85,1.6)  node[above=1pt]{$E^{+}$};
\draw [ very thick, ->] (-3,0) -- (3,0) node[anchor=south] {$\mathbb{Z}$};
\draw[red,very thick] (-1.4,0)--(-1.2,0.05)--(-1,0.3)--(-0.8,0.35)--(-0.65,0.5)--(-0.4,0.55)--(-0.3,0.6)--(-0.2,0.9)--(0,0.93)--(0.03,1)--(0.1,1.3)--(0.2,1.4)--(0.7,1.45)--(0.85,1.6);
\draw[blue,very thick] (0,-1.2)--(0.2,-1)--(0.3,-0.96)--(0.34,-0.9)--(0.4,-0.8)--(0.5,-0.6)--(0.75,-0.52)--(0.85,-0.4)--(1,-0.2)--(1.1,-0.1)--(1.2,-0.02)--(1.3,0.1)--(1.5,0.4)--(1.7,0.45)--(1.8,0.6)--(1.9,0.9)--(1.95,1)--(2.1,1.3)--(2.2,1.4)--(2.4,1.7)--(2.6,1.9)--(2.75,2)--(2.9,2.5);

\end{tikzpicture}
\caption{
We choose $k=k(a)$ such that $\lim_{a\to \infty}k(a)=\infty,\lim_{a \to \infty}\frac{k(a)}{a}=0,\varepsilon(a)>\frac{k}{a}.$ 
Then, the  maximizing path (blue) from $\mathcal{L}^{-}=(0,-\lfloor a t^{2/3}\rfloor)$ to  $E=(\lfloor t + \frac{u}{a}t^{2/3}\rfloor,\lfloor t \rfloor)$ crosses  
the line segment $R_{-}=R_{-}(k)$ (dotted) with vanishing probability  as $\lim_{a \to \infty} \lim_{t \to \infty}$. The point $E^{+}$ is  at distance $\varepsilon(a)t$ from $E$ on the line connecting $E$ with $\mathcal{L}^{+}=(-\lfloor at^{2/3}\rfloor,0)$ (see \eqref{E+}).
   The maximizer from $\mathcal{L}^{+} $ to $ E^{+}$ crosses $R_{+}=R_{+}(k)$ (dashed) with vanishing probability.  So the two maximizers do not cross  asymptotically, leading to the decoupling.}\label{firt}
\end{figure}
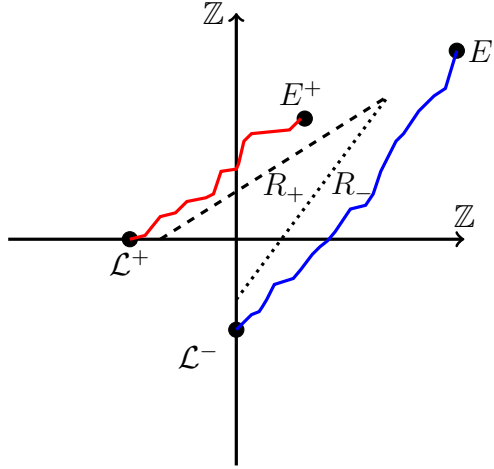

Next we choose the point $E^{+}$ from Assumption 2. From Proposition \ref{JohConvergence} one can easily compute  that $E^{+}$ should lie on the line segment from $\mathcal{L}^{+}$ to $E=(\lfloor t+\frac{u}{a}t^{2/3}\rfloor,\lfloor t\rfloor )$, so it remains to choose $\varepsilon(a)$.
To motivate this choice, note that by Theorem \ref{trans1} we can control the probability that 
$\pi^{\mathrm{max}}_{\mathcal{L}^{+} \to E}, \pi^{\mathrm{max}}_{\mathcal{L}^{-} \to E} $ have transversal fluctuations of order $k t^{2/3}$. In particular, we have  a good upper bound for the probability that $\pi^{\mathrm{max}}_{\mathcal{L}^{+} \to E}$ contains no point of the straight   line  $R_{+}$ which joins  (in $\Z^{2},$ see \eqref{R+}) the points 
  $(\lfloor -at^{2/3}+kt^{2/3}\rfloor,0)$ and $E+(\lfloor kt^{2/3}\rfloor ,0)$ and for the probability that  $\pi^{\mathrm{max}}_{\mathcal{L}^{-} \to E}$ 
  contains  no point of the straight line  $R_{-}$ joining  $(0,\lfloor -at^{2/3}+kt^{2/3}\rfloor )$ and $ E+(0,\lfloor kt^{2/3}\rfloor )$.
Now an elementary calculation reveals that  $R_{-}$ and $R_+$ cross in a point 
\begin{equation}\label{crosspoint}
\left(\left \lfloor t\left(1-\frac{k}{a}\right)+\mathcal{O}(t^{2/3})\right\rfloor, \left\lfloor t\left(1-\frac{k}{a}\right)+\mathcal{O}(t^{2/3})\right\rfloor\right),
\end{equation}
see Figure \ref{firt}.

In view of Assumption \ref{Assumpt3}, we thus should choose $\varepsilon(a)>\frac{k}{a},$ though  to satisfy Assumption \ref{Assumpt2}, this is not necessary, as the following result shows.

\begin{prop}\label{propA2}
Let $1>\varepsilon(a)>0.$   Then Assumption  \ref{Assumpt2} holds with
\begin{align}
&\label{E+}E^{+}=(\lfloor t(1-\varepsilon(a))+t^{2/3}(u/a-\varepsilon(a)(u/a+a))\rfloor,\lfloor t(1-\varepsilon(a))\rfloor)
\\&\mu^{\varepsilon(a)}t=4 \varepsilon(a)t+2\varepsilon(a)(u/a+a)t^{2/3}-\frac{\varepsilon(a)(a+u/a)^{2}}{4}t^{1/3}
\\&c_{\varepsilon(a)}=(1-\varepsilon(a))^{-1/3}.
\\&G_{0}^{a}(s)=F_{\GUE}(s \varepsilon(a)^{-1/3}).
\end{align}
\end{prop}
\begin{proof} We have  $L_{\mathcal{L}^{+} \to E^{+}}=^{d}L_{0 \to ((\lfloor t(1-\varepsilon(a))+t^{2/3}r_1 \rfloor,\lfloor t(1-\varepsilon(a))\rfloor)}$ for  $r_1 = (u/a+a)(1-\varepsilon(a))$. The $\mu_{\rm pp}$ of Proposition  \ref{JohConvergence} for 
$L_{\mathcal{L}^{+} \to E^{+}}$ is given by  $\mu_{\rm pp}t=4t(1-\varepsilon(a))+2 t^{2/3}r_1-\frac{r_{1}^{2}}{4(1-\varepsilon(a))}t^{1/3}$ and the one, with $E=(\lfloor t+\frac{u}{a}t^{2/3}\rfloor,\lfloor t\rfloor )$,  of $L_{E^{+} \to E}$ 
equals 
$\mu^{\varepsilon(a)}t=4t\varepsilon(a)+2 t^{2/3}r_2-\frac{r_{2}^{2}}{4\varepsilon(a)}t^{1/3},$ with    $r_2=-r_1+u/a+a$ and   since the two terms need to sum up to  $\mu^{a}$ from  Theorem \ref{Thmpp}  we obtain the condition
\begin{equation}
\frac{r_{2}^{2}}{4\varepsilon(a)}+\frac{r_{1}^{2}}{4(1-\varepsilon(a))}=\frac{(u/a+a)^{2}}{4},
\end{equation}
which is precisely solved by our $r_1.$ Finally, $c_{\varepsilon(a)}$ and $G_{0}^{a}$ are immediately obtained from Proposition  \ref{JohConvergence}.
\end{proof}

Let now $E^{+} $ be as in \eqref{E+} and denote $E^{+,k}=E^{+}+(k t^{2/3},0)$. Define
$
\mathfrak{R}_{+}(k)=\overline{(\lfloor - a t^{2/3}+kt^{2/3}\rfloor,0)E^{+,k}}
$
as the line segment  (in $\R^{2}$ ) from $(\lfloor - a t^{2/3}+kt^{2/3}\rfloor,0)$ to $
E^{+,k}$, and denote \begin{equation}\label{R+}R_{+}(k)=\{x \in \Z^{2}: |x-y|\leq 2 \mathrm{\,\,for \,\,a \,\,} 
y \in \mathfrak{R}_{+}(k)\}\end{equation} a discrete approximation. See Figure \ref{firt}.
 Denote by $\Pi^{+,k}$ the set of up-right paths from $\mathcal{L}^{+}$ to $E^{+}$ which do not contain any point of $R_{+}(k)$. Set 
\begin{equation}\label{tilde+}
\tilde{L}_{ \mathcal{L}^{+} \to E^{+}}=\tilde{L}_{ \mathcal{L}^{+} \to E^{+}}(k)=\max_{\pi \in \Pi^{+,k} }\sum_{(i,j) \in \pi}\omega_{i,j} .
\end{equation}
Let now $E=(\lfloor t +\frac{u}{a}t^{2/3}\rfloor, \lfloor t \rfloor )$ and $E^{k}=E+(0,kt^{2/3})$.
Write 
$
\mathfrak{R}_{-}(k)=\overline{(0,\lfloor-a t^{2/3}+kt^{2/3}\rfloor)E^{k}}
$
for the line segment in $\R^{2}$ joining $(0,\lfloor-a t^{2/3}+kt^{2/3}\rfloor)$ and $E^{k}$ and set
\begin{equation}\label{R-}
R_{-}(k)=\{ x \in \Z^{2}:|x-y|\leq 2 \mathrm{\, \,for \, \,a\,\,} y\in \mathfrak{R}_{-}(k)\}. 
\end{equation}
Define $\Pi^{-,k}$ to be the set of up-right paths from $\mathcal{L}^{-}$ to $E$ which do not contain any point of $R_{-}(k)$. We define
\begin{equation}\label{tilde-}
\tilde{L}_{\mathcal{L}^{-}\to E}=\tilde{L}_{\mathcal{L}^{-}\to E}(k)=\max_{\pi \in \Pi^{-,k}}\sum_{i,j \in \pi}\omega_{i,j}.
\end{equation}

\begin{prop}\label{propA3} Let $a>k>0$ and  let $E^{+}$ be given by  \eqref{E+} 
with $1>\varepsilon(a)>\frac{k}{a}$ and let $\tilde{L}_{\mathcal{L}^{-}\to E}(k), \tilde{L}_{ \mathcal{L}^{+} \to E^{+}}(k)$ be given by \eqref{tilde-},\eqref{tilde+}.
Then there are  constants $c,C>0$ such that Assumption \ref{Assumpt3} holds with $\psi=Ce^{-ck}$.
\end{prop}
\begin{proof}
Note that we have  $L_{\mathcal{L}^{+} \to E^{+}}=^{d}L_{0 \to ((\lfloor t(1-\varepsilon(a))+t^{2/3}r_1 \rfloor,\lfloor t(1-\varepsilon(a))\rfloor)}$ for  $r_1 = (u/a+a)(1-\varepsilon(a)).$ Write $(\lfloor t(1-\varepsilon(a))+t^{2/3}r_1 \rfloor,\lfloor t(1-\varepsilon(a))\rfloor)=(m_+ ,n_+).$
Thus  by translation invariance and Theorem  \ref{trans1} 
\begin{align}
&\limsup_{t \to \infty}\Pb(L_{\mathcal{L}^{+} \to E^{+}}\neq  \tilde{L}_{\mathcal{L}^{+} \to E^{+}}(k))
\\&\leq \limsup_{t \to \infty}\Pb(\max_{\kappa \in [0,1]}\{Z_{\lfloor \kappa n_+ \rfloor }(m_+,n_+)-\kappa m_+\}\geq k t^{2/3})\leq C e^{-ck}.
\end{align}
Furthermore, 
\begin{equation}
L_{\mathcal{L}^{-} \to E}=^{d}L_{0 \to (\lfloor t +\frac{u}{a}t^{2/3}\rfloor,\lfloor t\rfloor+\lfloor a t^{2/3} \rfloor)}.
\end{equation}
Setting $T=\lfloor t\rfloor+\lfloor a t^{2/3} \rfloor$ we have $t+\frac{u}{a}t^{2/3}=T+(\frac{u}{a}-a)T^{2/3}+\mathcal{O}(T^{1/3})=T+c_{-}(T)T^{2/3}$ for a $c_{-}(T)=\frac{u}{a}-a+\mathcal{O}(T^{-1/3})$ such that 
\begin{equation}
L_{\mathcal{L}^{-} \to E}=^{d}L_{0 \to (\lfloor T+c_{-}(T)T^{2/3}\rfloor,\lfloor T \rfloor)}.
\end{equation}
We thus get 
\begin{align*}
&\limsup_{ t \to \infty}\Pb(L_{\mathcal{L}^{-} \to E}\neq \tilde{L}_{\mathcal{L}^{-}\to E})\\&\leq \limsup_{ t \to \infty}\Pb \left(\max_{\kappa \in [0,1]}\{Y^{\mathrm{TOP}}_{\lfloor \kappa (T+c_{-}(T)T^{2/3})\rfloor}-\kappa T\}\geq kT^{2/3}/2 \right)
\\& \leq C e^{-ck}.
\end{align*}
Finally, the independence of $\tilde{L}_{\mathcal{L}^{-}\to E}(k), \tilde{L}_{ \mathcal{L}^{+} \to E^{+}}(k)$ follows from choosing $\varepsilon(a)>k/a$ and \eqref{crosspoint}:  The admissible paths for $\tilde{L}_{ \mathcal{L}^{+} \to E^{+}}(k)$ 
do not cross $R_{+}(k)$ from \eqref{R+}, and the admissible paths for $\tilde{L}_{\mathcal{L}^{-}\to E}(k)$  do not cross $R_{-}(k)$ from \eqref{R-}, and since $\varepsilon(a)>k/a,$  we have by \eqref{crosspoint} that 
$R_{+}(k), R_{-}(k)$ do not cross each other, see also  Figure \ref{firt}. So $\tilde{L}_{\mathcal{L}^{-}\to E}(k), \tilde{L}_{ \mathcal{L}^{+} \to E^{+}}(k)$  may only use points from disjoint, (deterministic) subsets of $\Z^{2},$ leading to the independence. 

\end{proof}

\begin{proof}[Proof of Theorem \ref{Thmpp}]
Assumptions \ref{Assumpt1},\ref{Assumpt2},\ref{Assumpt3} of Theorem \ref{GenGen} have been verified in  Propositions \ref{JohConvergence}, \ref{propA2},\ref{propA3}, such that the result follows.

\end{proof}

Next we proof  Theorem \ref{thmtasep}.

\begin{proof}[Proof of Theorem \ref{thmtasep}]
 Define $c_1= -\frac{\frac{u}{a}+a}{2},$  $c_{2}=\frac{u}{a}$ and $\xi_{2}=\frac{(u/a+a)^{2}}{2}-2^{-1/3}s$.
 Note that (see e.g. Theorem 5 in \cite{CP15}) for $K \in \N, v \in \R, \gamma \in [0,1/3]$
 \begin{equation}\label{local}
 \lim_{K \to \infty}\frac{L_{0 \to (K+\lfloor K^{\gamma} v\rfloor,K)}-L_{0 \to (K,K)}-2vK^{\gamma }}{K^{1/3}}=0.
 \end{equation}
In particular, since we are only interested in asymptotic results, any shift of order $1$ of the end/ starting point for a point-to-point LPP time will be asymptotically irrelevant.We set 
\begin{align*}
t=\left\lfloor \frac{T}{4}+c_{1}T^{2/3}\right\rfloor
\quad M=t+\lfloor c_{2}T^{2/3}+\xi_2 T^{1/3}\rfloor
\end{align*}
Then $T^{1/3}=(4t)^{1/3}+\mathcal{O}(1)$ as well as 
\begin{equation}
\begin{aligned}
&T=4t-c_{1}t^{2/3}4^{5/3}+c_{1}^{2}\frac{2}{3}t^{1/3}4^{7/3}+\mathcal{O}(1)
\\&T^{2/3}=(4t)^{2/3}-c_{1}\frac{2}{3}t^{1/3}4^{4/3}+\mathcal{O}(1).
\end{aligned}
\end{equation}
We define furthermore
\begin{equation}
\hat{\mathcal{L}}^{+}=(\lfloor -a((4t)^{2/3}-c_{1}\frac{2}{3}t^{1/3}4^{4/3})\rfloor,0)\quad \hat{\mathcal{L}}^{-}=(0,\lfloor -a((4t)^{2/3}-c_{1}\frac{2}{3}t^{1/3}4^{4/3})\rfloor)
\end{equation}
and $\hat{\mathcal{L}}=\hat{\mathcal{L}}^{+}\cup \hat{\mathcal{L}}^{-}$.
Then by the link \eqref{LPPTASEP}
\begin{equation}
\begin{aligned}
&\lim_{T \to \infty}\Pb \bigg( x_{\lfloor \frac{T}{4}-T^{2/3}\frac{a+\frac{u}{a}}{2}\rfloor}(T)\geq \frac{u}{a}T^{2/3}+T^{1/3}\frac{(\frac{u}{a}+a)^{2}}{2}-\frac{T^{1/3}}{2^{1/3}}s \bigg)
\\&=\lim_{T \to \infty}\Pb \left(L_{\{(-\lfloor a T^{2/3}\rfloor,0),(0,-\lfloor a T^{2/3}\rfloor )\}\to (\lfloor \frac{T}{4}+c_{1}T^{2/3}+c_{2}T^{2/3}+\xi_{2}T^{1/3}\rfloor,\lfloor \frac{T}{4}+c_{1}T^{2/3}\rfloor)}\leq T\right)
\\&=\lim_{t \to \infty}\Pb \left(L_{\hat{\mathcal{L}}\to (M, t)}\leq 4t-c_{1}t^{2/3}4^{5/3}+c_{1}^{2}\frac{2}{3}t^{1/3}4^{7/3}\right).
\end{aligned}
\end{equation}
We now check the Assumptions 1,2,3 for the LPP times $L_{\hat{\mathcal{L}}^{+}\to (M,t)},L_{\hat{\mathcal{L}}^{-}\to (M,t)}$.
By Proposition \ref{JohConvergence} and \eqref{local}, we have with $\mu_{\hat{\mathcal{L}}^{+}\to (M,t)},\mu_{\hat{\mathcal{L}}^{-}\to (M,t)}$ defined by 
\begin{equation}
\begin{aligned}
&\mu_{\hat{\mathcal{L}}^{+}\to (M,t)}t=4t+2(c_2+a)(4t)^{2/3}-4^{1/3}(c_2+a)^{2}t^{1/3}+2(4t)^{1/3}(\xi_2-8c_{1}(c_2+a)/3)
\\&\mu_{\hat{\mathcal{L}}^{-}\to (M,t)}t=\mu_{\hat{\mathcal{L}}^{+}\to (M,t)}t+u4^{4/3}t^{1/3}
\end{aligned}
\end{equation}
the convergence
\begin{equation}
\begin{aligned}\label{GUEhat}
&\lim_{t \to \infty}\Pb\left(L_{\hat{\mathcal{L}}^{+}\to (M,t)}\leq\mu_{\hat{\mathcal{L}}^{+}\to (M,t)}t+ s2^{4/3}t^{1/3}\right)=F_{\GUE}(s)
\\&\lim_{t \to \infty}\Pb\left(L_{\hat{\mathcal{L}}^{-}\to (M,t)}\leq\mu_{\hat{\mathcal{L}}^{-}\to (M,t)}t+ s2^{4/3}t^{1/3}\right)=F_{\GUE}(s).
\end{aligned}
\end{equation}
The choice of $c_1,c_2,\xi_2$ is precisely such that 
\begin{equation}
\mu_{\hat{\mathcal{L}}^{+}\to (M,t)}t+ s2^{4/3}t^{1/3}=4t-c_{1}t^{2/3}4^{5/3}+c_{1}^{2}\frac{2}{3}t^{1/3}4^{7/3}.
\end{equation}

So \eqref{GUEhat} verifies Assumption 1. Next we choose the point $\hat{E}^{+}$ of Assumption 2. Note that with $\tilde{a}=a4^{2/3},\tilde{u}=4^{4/3}u$  we have
\begin{align}\label{t13}
L_{\hat{\mathcal{L}}^{+}\to (M,t)}=L_{(-\lfloor \tilde{a} t^{2/3}+\mathcal{O}(t^{1/3})\rfloor,0)\to (\lfloor t+\frac{\tilde{u}}{\tilde{a}}t^{2/3}+\mathcal{O}(t^{1/3})\rfloor ,t)}.
\end{align}
This is, with $\tilde{a},\tilde{u}$ instead of $a,u$ and  up to an $\mathcal{O}(t^{1/3})$ horizontal shift in the starting and end point, the same LPP time for which we chose $E^{+}$ in \eqref{E+}.   Hence we can take $\hat{E}^{+}$ as $E^{+}$ in \eqref{E+}, only with $\tilde{a}, \tilde{u}$ instead of $a,u$.
Finally, Assumption 3 can be verified as in the proof of Theorem \ref{Thmpp}, the horizontal  $\mathcal{O}(t^{1/3})$ shifts in \eqref{t13} (and in $L_{\hat{\mathcal{L}}^{-}\to (M,t)}$ ) not affecting the argument.
\end{proof}
To prove Corollary \ref{CorA2}, we need to know how the $\mathrm{Airy}_{2}$ arises in the LPP model of Theorem \ref{Thmpp}, hence the following lemma.
\begin{lem}\label{airy22} Let $\mathcal{L}$ be as in Theorem \ref{Thmpp}.
We have 
\begin{align}
\lim_{t \to \infty}\Pb\left(\frac{L_{\mathcal{L} \to (t+\frac{u}{a}t^{2/3},t)}-\mu^{a}t}{2^{4/3}t^{1/3}}\leq s\right)=\Pb\bigg(&\mathcal{A}_{2}\left(\frac{-a-\frac{u}{a}}{2^{5/3}}\right)\leq s, \\&\mathcal{A}_{2}\left(\frac{a-\frac{u}{a}}{2^{5/3}}\right)\leq s-\frac{u}{2^{4/3}}\bigg).
\end{align}
\end{lem}
\begin{proof}
By exchanging the end point and $\mathcal{L}$ we see that 
\begin{equation}\label{LPPairy2}
L_{\mathcal{L} \to (t+\frac{u}{a}t^{2/3},t)}=^{d}\max\{L_{(0,0) \to (t+(u/a+a)t^{2/3},t)},L_{(0,0) \to (t+ut^{2/3}/a,t+at^{2/3})}\}
\end{equation}
where $=^{d}$ denote equality in distribution.  Now, for end points lying on the same horizontal line, the convergence of the rescaled LPP times to the $\mathcal{A}_{2}$ process is e.g. proven in Corollary 2.4 in \cite{FO17}. To extend this result to the two endpoints of \eqref{LPPairy2} (which do not lie on the same horizontal line), one follows the straight (characteristic)  line joining $(0,0)$ and $ (t+(u/a+a)t^{2/3},t)$ until it reaches the horizontal line $(\cdot,t+at^{2/3})$  in a point $P$. By slow decorrelation,  we may replace  $L_{(0,0) \to (t+(u/a+a)t^{2/3},t)}$ by $L_{(0,0) \to P}$  without altering the limiting distribution, giving the result.
\end{proof}
\begin{proof}[Proof of Corollary \ref{CorA2}]
It is an immediate Corollary of Theorem \ref{Thmpp} and Lemma \ref{airy22}, by a simple change of variable.

\end{proof}

\subsection{Proof of Theorems \ref{twotime} and  \ref{GenGen}  }\label{222}

Next we come to the proof of Theorem  \ref{twotime}.
\begin{proof}[Proof of Theorem \ref{twotime}]
The lower bound in \eqref{2time} follows from the FKG inequality and the known convergence to the $\mathrm{Airy}_2$ process, see Theorem 2 in \cite{BP07}. For the upper bound, define the points
\begin{equation}
P_{2}(u)=(\lfloor \tau t +u t^{2/3}(\tau a^{-1/3}-a^{2/3}) \rfloor, \lfloor \tau t+1\rfloor), 
\end{equation}
and set $\mu_{u}t=4\tau t+2\tau t^{2/3}ua^{-1/3}-u^{2}\frac{\tau t^{1/3} }{4a^{2/3}}.$ Then for any $\delta >0$ 
\begin{equation}\label{kgue}
\lim_{t \to \infty}\Pb\left(\bigcup_{i=1}^{k}\frac{L_{P(a,u_i)\to  P_{2}(u_i)}-\mu_{u_i}t}{2^{4/3}(at)^{1/3}}\leq -\delta \right)\leq kF_{\GUE}(-\delta a^{1/3}\tau^{-1/3}).
\end{equation}
Denote for brevity $\mathcal{F}=\bigcap_{i=1}^{l}\{L^{\mathrm{resc}}_{P(\tau,r_i)\to (\lfloor ct_j\rfloor ,\lfloor \tau t_j\rfloor )}\leq s_i\}$. Then, using subadditivity and \eqref{kgue}, we get 
\begin{align}
& \lim_{ t_j \to \infty}\Pb\left( \mathcal{F}\cap  \bigcap_{i=1}^{k}\{L^{\mathrm{resc}}_{P(a,u_i)\to (\lfloor at_j\rfloor ,\lfloor at_j \rfloor )}\leq \zeta_i\}\right)
 \\&\leq \lim_{ t_j \to \infty}\Pb\left( \mathcal{F}\cap  \bigcap_{i=1}^{k}\{  
 \frac{L_{P(a,u_i)\to  P_{2}(u_i)}-\mu_{u_i}t_j}{2^{4/3}(at_j)^{1/3}}+ \frac{L_{ P_{2}(u_i)\to (\lfloor at_j \rfloor, \lfloor at_j \rfloor)}-\mu(a,u_i)t_j+\mu_{u_i}t_j}{2^{4/3}(at_j)^{1/3}}
  \leq \zeta_i\}\right)
  \\&\label{eqqq}\leq \lim_{ t_j \to \infty}\Pb\left( \mathcal{F}\cap  \bigcap_{i=1}^{k}\{  
  \frac{L_{ P_{2}(u_i)\to (\lfloor at_j \rfloor, \lfloor at_j \rfloor)}-\mu(a,u_i)t_j+\mu_{u_i}t_j}{2^{4/3}(at_j)^{1/3}}
  \leq \zeta_i+\delta \}\right)\\&\label{eqqqq} + kF_{\GUE}(-\delta a^{1/3}\tau^{-1/3}).
  \end{align}
  Note now that $L_{P(\tau,r)\to (\lfloor \tau t\rfloor,\lfloor \tau t \rfloor)}$ and $L_{ P_{2}(u)\to (\lfloor at \rfloor, \lfloor at\rfloor)}$ are independent for all $r,u \in \R$. Hence we get that 
  \begin{align}
  \eqref{eqqq}&=\lim_{ t\to \infty}\Pb\left( \mathcal{F}\right)\Pb\left(  \bigcap_{i=1}^{k}\{  
  \frac{L_{ P_{2}(u_i)\to (\lfloor at \rfloor, \lfloor at \rfloor)}-\mu(a,u_i)t+\mu_{u_i}t}{2^{4/3}(at)^{1/3}}
  \leq \zeta_i+\delta \}\right)  
  \\&= \Pb\left(\bigcap_{i=1}^{l}\mathcal{A}_{2}(r_i)\leq s_i\right)\Pb\left(\bigcap_{i=1}^{k}\mathcal{A}_{2}(u_i (1-\tau/a)^{1/3})\leq (\zeta_i+\delta)\frac{a^{1/3}}{(a-\tau)^{1/3}} \right),
\end{align}
finishing the proof.
  
\end{proof}
We conclude by proving the general Theorem \ref{GenGen}.
\begin{proof}[Proof of Theorem \ref{GenGen}]
For  the lower bound, note  that $\{L_{\mathcal{L}^{+}\to E_1}^{\mathrm{resc}}\leq s_1\}, \{L_{\mathcal{L}^{-}\to E_2}^{\mathrm{resc}}\leq s_2\}$ are decreasing events. 
Thus by the FKG inequality we have 
\begin{equation}
\Pb(L_{\mathcal{L}^{+}\to E_1}^{\mathrm{resc}}\leq s_1)\Pb( L_{\mathcal{L}^{-}\to E_2}^{\mathrm{resc}}\leq s_2)\leq \Pb(L_{\mathcal{L}^{+}\to E_1}^{\mathrm{resc}}\leq s_1, L_{\mathcal{L}^{-}\to E_2}^{\mathrm{resc}}\leq s_2).
\end{equation}
Now by  Assumption \ref{Assumpt1} we have 
\begin{equation}
\lim_{t \to \infty}\Pb(L_{\mathcal{L}^{+}\to E_1}^{\mathrm{resc}}\leq s_1)\Pb( L_{\mathcal{L}^{-}\to E_2}^{\mathrm{resc}}\leq s_2)=G_{1}^{a}(s_1)G_{2}^{a}(s_2),
\end{equation}
which proves the lower bound.

For the upper bound, denote  $A^{\delta}=\{L_{E^{+}\to E_1}^{\mathrm{resc}}\leq -\delta\}.$ 
Noting that $L_{\mathcal{L}^{+}\to E_1}^{\mathrm{resc}}\geq L_{E^{+}\to E_1}^{\mathrm{resc}}+  L_{\mathcal{L}^{+}\to E^{+}}^{\mathrm{resc}}$
we get from Assumptions \ref{Assumpt1},\ref{Assumpt2} and  \ref{Assumpt3}

\begin{equation}
\begin{aligned}
&\lim_{t_k \to \infty}\Pb(L_{\mathcal{L}^{+}\to E_1}^{\mathrm{resc}}\leq s_1, L_{\mathcal{L}^{-}\to E_2}^{\mathrm{resc}}\leq s_2)
\\&\leq 
\lim_{t_k \to \infty}
\Pb(\{L_{E^{+}\to E_1}^{\mathrm{resc}}+  L_{\mathcal{L}^{+}\to E^{+}}^{\mathrm{resc}}\leq s_1\}\cap\{ L_{\mathcal{L}^{-}\to E_2}^{\mathrm{resc}}\leq s_2\}\cap (A^{\delta}\cup (A^{\delta})^c))
\\&\leq G_{0}^{a}(-\delta)+\lim_{t_k \to \infty}\Pb(\{-\delta + L_{\mathcal{L}^{+}\to E{+}}^{\mathrm{resc}}\leq s_1\}\cap \{ L_{\mathcal{L}^{-}\to E_2}^{\mathrm{resc}}\leq s_2\}\cap
(A^{\delta})^c)
\\& 
\leq G_{0}^{a}(-\delta)+
\lim_{t_k \to \infty}\Pb(\{ L_{\mathcal{L}^{+}\to E^{+}}^{\mathrm{resc}}\leq s_1+\delta \}\cap \{ L_{\mathcal{L}^{-}\to E_2}^{\mathrm{resc}}\leq s_2\})
\\& \leq G_{0}^{a}(-\delta)+\tilde{\psi}+
\lim_{t_k \to \infty}\Pb(\{ \tilde{L}_{\mathcal{L}^{+}\to E^{+}}^{\mathrm{resc}}\leq s_1+\delta \})\Pb(\{ \tilde{L}_{\mathcal{L}^{-}\to E_2}^{\mathrm{resc}}\leq s_2\})
\\&\leq  G_{0}^{a}(-\delta)+3\tilde{\psi}+
\lim_{t_k \to \infty}\Pb(\{ L_{\mathcal{L}^{+}\to E^{+}}^{\mathrm{resc}}\leq s_1+\delta \})\Pb(\{ L_{\mathcal{L}^{-}\to E_2}^{\mathrm{resc}}\leq s_2\})
\\&= G_{0}^{a}(-\delta)+3\tilde{\psi}+G_{1}^{a}((s_1+\delta)c_{\varepsilon(a)})G_{2}^{a}(s_2).
\end{aligned}
\end{equation}
\end{proof}

\section{Proof for line-to-point problems : Theorems \ref{shockflat} and \ref{airydec} }\label{othgen} In this Section we prove  Theorem \ref{shockflat} as well as  Theorem \ref{airydec}.

We also  discuss   the transition to shock fluctuations when for $a=0$  one has flat (deterministic) initial data.

We start with  the proof of Theorem  \ref{airydec}. The proof is summarized in Figure 2, the actual proof is slightly technical.

\begin{figure}[h]
  \centering
\begin{tikzpicture}

\draw [ very thick, ->] (0,-3) -- (0,3) node[anchor=east] {$\mathbb{Z}$};
\draw [  thick] (2.6,-2.6) -- (-2.6,2.6) node[anchor=east] {$\mathcal{L}$};
\draw [red, very thick,<->] (0.3,-0.3) -- (-0.3,0.3) ;
\draw [blue, very thick,<->] (-0.5,0.5) -- (-1.1,1.1) ;

\draw  (-1.6,0.45)  node[above=1pt]{$\mathcal{F}_{2}(k)$};
\draw  (-0.7,-0.7)  node[above=1pt]{$\mathcal{F}_{1}(k)$};

\draw  (1.2,2.8)  node[above=1pt]{$E_{2}$};
\draw  (1.9,2)  node[above=1pt]{$E_{1}$};
\filldraw     (2,2)      circle (0.1 cm);
\filldraw     (1.2,2.8)      circle (0.1 cm);
\draw [ very thick, ->] (-3,0) -- (3,0) node[anchor=south] {$\mathbb{Z}$};
\begin{scope}[xshift=1.4cm]
\draw[red,very thick,dashed] (-1.4,0)--(-1.2,0.05)--(-1,0.3)--(-0.8,0.35)--(-0.65,0.5)--(-0.4,0.55)--(-0.3,0.6)--(-0.2,0.9)--(0,0.93)--(0.03,1)--(0.1,1.3)--(0.2,1.4)--(0.3,1.45)--(0.6,2);
\end{scope}
\begin{scope}[xshift=-0.8cm,yshift=2cm]
\draw[blue,very thick,dashed] (0,-1.2)--(0.2,-1)--(0.3,-0.96)--(0.34,-0.9)--(0.4,-0.8)--(0.5,-0.6)--(0.75,-0.52)--(0.85,-0.4)--(1,-0.2)--(1.1,-0.1)--(1.2,-0.02)--(1.3,0.1)--(1.5,0.4)--(1.7,0.45)--(1.8,0.6)--(2,0.9);
\end{scope}
\end{tikzpicture}
\caption{Proving \eqref{decouplage}: The maximizer (blue, dashed)  $\pi_{\mathcal{L}\to E_2}^{\mathrm{max}}$  from $\mathcal{L}=\{(-k,k):k\in \Z\}$ to $E_{2}=(t-at^{2},t+at^{2/3})$ starts with high probability in the line segment $F_{2}(k)$ (blue, see \eqref{FF}). In this event, the transversal fluctuations of $\pi_{\mathcal{L}\to E_2}^{\mathrm{max}}$ are  bounded by those of the maximizer starting at the bottom right end point  of  $F_{2}(k)$ and going to $E_2$, which themselves are bounded by Theorem \ref{trans1}. The same argument  applies to $\pi_{\mathcal{L}\to E_1}^{\mathrm{max}}$ (red,dashed) with $E_{1}=(t,t)$.  This shows $\pi_{\mathcal{L}\to E_1}^{\mathrm{max}},\pi_{\mathcal{L}\to E_2}^{\mathrm{max}}$  stay in disjoint sets with high probability, which together with convergence to the  $\mathrm{Airy}_{1}$ process of the LPP time, implies \eqref{decouplage}. 
}
\end{figure}
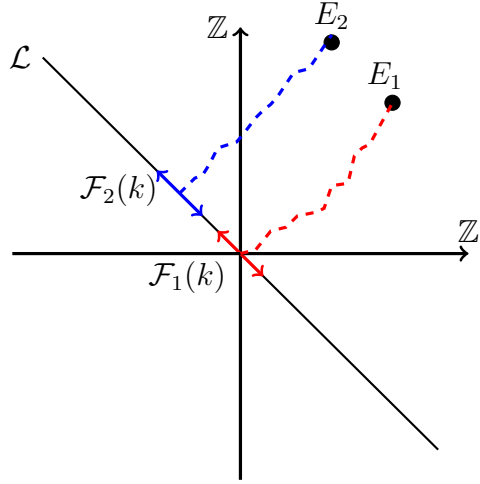

\begin{proof}[Proof of Theorem \ref{airydec}]
We start by proving \eqref{decouplage}.
We apply Theorem \ref{GenGen} with 
 $\mathcal{L}= \mathcal{L}^{+}=\mathcal{L}^{-}=\{  (-k,k):k\in \Z  \} , E_{+}=E_1=(\lfloor t \rfloor ,\lfloor t\rfloor ), E_{2}=(\lfloor t\rfloor -\lfloor at^{2/3}\rfloor ,\lfloor t\rfloor +\lfloor at^{2/3}\rfloor ).$ One obtains from Theorem 2.2 of \cite{BF07} and the link between TASEP and LPP (all weights i.i.d., $\omega_{i,j}\sim \exp(1)$)
\begin{equation}\label{airy1}
\lim_{t \to \infty} \Pb(\cap_{i=1}^{2}\{L_{\mathcal{L}\to E_i}\leq 4t +s_i t^{1/3}\})=\Pb(\mathcal{A}_1 (0) \leq 2^{-5/3}s_1,\mathcal{A}_1 (a4^{-2/3}) \leq 2^{-5/3}s_2).
\end{equation}
Also for $i=1,2$ we have
\begin{equation}\label{GOE}
\lim_{ t\to \infty}\Pb(L_{\mathcal{L}\to E_i}\leq 4t+st^{1/3})=F_{\GOE}(2^{-2/3}s).
\end{equation} 
Now \eqref{GOE} shows Assumption \ref{Assumpt1}, and Assumption \ref{Assumpt2}
is trivially fullfilled with  $c_{\varepsilon(a)}=1, G_{0}^{a}=\mathbf{1}_{[0, \infty)}$.
Set now for $k \in \Z$ \begin{align}
&\mathcal{F}^{1}(k)=\{(- i , i ),i=-\lfloor kt^{2/3}\rfloor ,\ldots,\lfloor kt^{2/3}\rfloor\},
\\&\mathcal{F}^{2}(k)=\{(-\lfloor at^{2/3}\rfloor -i ,\lfloor a t^{2/3}\rfloor + i ),i=-\lfloor kt^{2/3}\rfloor ,\ldots,\lfloor kt^{2/3}\rfloor \}.\label{FF}
\end{align}
Denote by $\pi^{\mathrm{max}}_{\mathcal{L}\to E_i}$ the maximizing path from 
$\mathcal{L}$ to $E_i$ and by  $\pi^{\mathrm{max}}_{\mathcal{L}\to E_i}(0)$ the point of $\pi^{\mathrm{max}}_{\mathcal{L}\to E_i}$ which belongs to $\mathcal{L}$. By a simple shift one sees $\Pb(\pi^{\mathrm{max}}_{\mathcal{L}\to E_i}(0)\in \mathcal{F}^{i}(k))$ is the same for $i=1,2$. Consequently, by (4.18) of \cite{FO17} one gets that for $k$ sufficiently large and some constants $C,c$
\begin{equation}\label{pi}
\Pb(\cup_{i=1}^{2}\{\pi^{\mathrm{max}}_{\mathcal{L}\to E_i}(0)\notin \mathcal{F}^{i}(k)\})\leq Ce^{-ck^{2}}.
\end{equation}
Let $E_3= \left(-\left\lfloor \frac{a}{4}t^{2/3}\right\rfloor ,\left\lfloor \frac{a}{4}t^{2/3}\right\rfloor\right)$ and $E_4=E_3+E_1.$  Define also $ E_5=\left(-\left\lfloor \frac{3a}{4}t^{2/3}\right\rfloor ,\left\lfloor \frac{3a}{4}t^{2/3}\right\rfloor\right),E_6=E_5+E_1.$ Denote by $R_{1}(k)$ the straight line (in $\Z^{2}$) which goes through  $E_3+(-\lfloor k t^{2/3}\rfloor, \lfloor k t^{2/3}\rfloor)$  and $E_4+(-\lfloor k t^{2/3}\rfloor, \lfloor k t^{2/3}\rfloor)$ and by $R_{2}(k)$ the straight line going through $E_5+(\lfloor k t^{2/3}\rfloor, -\lfloor k t^{2/3}\rfloor)$ and $E_6+(\lfloor k t^{2/3}\rfloor, -\lfloor k t^{2/3}\rfloor)$. Denote by $U_1$ the event
$\{\pi^{\mathrm{max}}_{E_3 \to E_4}\cap R_{1}(a/10)=\emptyset\}$ and by $U_2$ the event $\{\pi^{\mathrm{max}}_{E_5 \to E_6}\cap R_{2}(a/10)=\emptyset\}.$

 It follows from Theorem \ref{trans1}  that 
\begin{equation}\label{U}
\limsup_{t \to \infty}\Pb(\cup_{i=1}^{2}U_{i}^{c})\leq  Ce^{-c a}.
\end{equation}

Now the event 
\begin{equation}
\bigcap_{i=1}^{2}\{\pi^{\mathrm{max}}_{\mathcal{L}\to E_i}(0)\in \mathcal{F}^{i}(a/10)\}\cap \bigcap_{i=1}^{2}U_i
\end{equation}
is a subset of 
\begin{equation}
\bigcap_{i=1}^{2}\{\pi^{\mathrm{max}}_{\mathcal{L}\to E_i}\cap R_{i}(a/10)=\emptyset\}
\end{equation}
(taking $a,t$ sufficiently large).
 Denote now for $i=1,2$ by $\Pi^{i}$ the set of up-right paths from $\mathcal{L}$ to $E_i$ which contain no point of $R_{i}(a/10).$ Set 
 \begin{equation}
 \tilde{L}_{\mathcal{L}\to E_i}=\max_{\pi \in \Pi^{i}}\sum_{(i,j)\in \pi}\omega_{i,j}.
  \end{equation}
  Note that $ \tilde{L}_{\mathcal{L}\to E_i},i=1,2$ are independent, and 
 by \eqref{pi}, \eqref{U} Assumption \ref{Assumpt3} is fullfilled with 
 \begin{equation}\label{psi(aa)}
 \tilde{\psi}=\tilde{\psi}(a)=Ce^{-c a^{2}}+ Ce^{-ca}.
 \end{equation} This finishes the proof.

 Next we come to the proof of \eqref{decouplage2}. Set $\mathcal{L}^{\mathrm{half}}=\{(-k,k):k\geq 0\}$. It follows from 
 Theorem 2 of \cite{BFS07} that with $E(k)=(\lfloor t -k t^{2/3}\rfloor,\lfloor t +k t^{2/3}\rfloor)$ and $b_1,b_2 \in \R$
 \begin{equation}
  \begin{aligned}\label{ass1a21}
& \lim_{t \to \infty}\Pb(\cap_{i=1}^{2}\{L_{\mathcal{L}^{\mathrm{half}}\to E(b_i)}\leq 4t+(s_i-2^{4/3}\min\{0,b_i\}^{2})t^{1/3}\})
 \\&=\Pb(\mathcal{A}_{2 \to 1}(b_12^{-2/3})\leq 2^{-4/3} s_1, \mathcal{A}_{2 \to 1}(b_2 2^{-2/3})\leq 2^{-4/3} s_2)«
 \end{aligned}
 \end{equation}
 To localize $\pi^{\mathrm{max}}_{\mathcal{L}^{\mathrm{half}}\to E(|b|+a)}(0)$ note that by a simple coupling, with $E_7= (-\lfloor (|b|+a)t^{2/3}\rfloor,\lfloor (|b|+a)t^{2/3}\rfloor)$  and \begin{equation*}\mathcal{F}^{3}(k)=\{E_7+(-i,i) :i=-\lfloor k t^{2/3}\rfloor,\ldots,\lfloor kt^{2/3}\rfloor\}\end{equation*}
 we have 
$ \{\pi^{\mathrm{max}}_{\mathcal{L}\to E(|b|+a)}(0)\in \mathcal{F}^{3}(a/10)\}\subseteq  \{\pi^{\mathrm{max}}_{\mathcal{L}^{\mathrm{half}}\to E(|b|+a)}(0)\in \mathcal{F}^{3}(a/10)\}$
such that $\Pb( \{\pi^{\mathrm{max}}_{\mathcal{L}^{\mathrm{half}}\to E(|b|+a)}(0)\notin \mathcal{F}^{3}(a/10)\})\leq Ce^{-ca^{2}}$ by (4.18) in \cite{FO17}.
Similarily, one can control $\Pb(\pi^{\mathrm{max}}_{\mathcal{L}^{\mathrm{half}}\to E(|b|+a/5)}(0)\notin  \mathcal{F}^{4}(a/20))\leq Ce^{-ca^{2}},$ where 
$ \mathcal{F}^{4}(k)=\{E_8+(-i,i):i=-\lfloor k t^{2/3}\rfloor, \ldots,\lfloor k t^{2/3}\rfloor\}$ with  $E_8=(-\lfloor (|b|+a/5)t^{2/3}\rfloor, \lfloor  (|b|+a/5)t^{2/3}\rfloor)$.
Let $R_{3}(k) $ be the line which connects $E_8+(-\lfloor k t^{2/3}\rfloor, \lfloor k t^{2/3}\rfloor)$ with $E(|b|+a/5)+(-\lfloor k t^{2/3}\rfloor, \lfloor k t^{2/3}\rfloor)$
and $R_{4}(k)$  the line which connects $E_7+(\lfloor k t^{2/3}\rfloor,- \lfloor k t^{2/3}\rfloor)$ with $E(|b|+a)+(\lfloor k t^{2/3}\rfloor,- \lfloor k t^{2/3}\rfloor).$
As was done above, we can bound 
\begin{equation}
\begin{aligned}\label{a221}
&\Pb( \pi^{\mathrm{max}}_{\mathcal{L}^{\mathrm{half}}\to E(|b|+a/5)}\cap R_{3}(a/10)\neq \emptyset)\leq Ce^{-c a^{2}}+Ce^{-ca}
\\&\Pb( \pi^{\mathrm{max}}_{\mathcal{L}^{\mathrm{half}}\to E(|b|+a)}\cap R_{4}(a/10)\neq \emptyset)\leq Ce^{-c a^{2}}+ Ce^{-ca}.
\end{aligned}
\end{equation}
Note now that if $  \pi^{\mathrm{max}}_{\mathcal{L}^{\mathrm{half}}\to E(|b|+a/5)}$ contains no point of $R_{3}(k)$ then this is also true for $  \pi^{\mathrm{max}}_{\mathcal{L}^{\mathrm{half}}\to E(b)}$. Let $\Pi^{3}$ be the up-right paths from $\mathcal{L}^{\mathrm{half}}$ to $E(b)$ which contain no point of $R_{3}(a/10)$, and $  \Pi^{4}$ be the up-right paths from $\mathcal{L}^{\mathrm{half}}$ to $E(|b|+a)$ which contain no point of $R_{4}(a/10)$. We define the independent random variables 
\begin{equation}
\tilde{L}_{\mathcal{L}^{\mathrm{half}}\to E(b)}=\max_{\pi \in \Pi^{3}}\sum_{(i,j)\in \pi}\omega_{i,j} \quad  \tilde{L}_{\mathcal{L}^{\mathrm{half}}\to E(|b|+a)}=\max_{\pi \in \Pi^{4}}\sum_{(i,j)\in \pi}\omega_{i,j}.
\end{equation}
Now we take $\mathcal{L}^{\mathrm{half}}=\mathcal{L}^{+}=\mathcal{L}^{-},E_{1}=E_{+}=E(b), E_{2}=E(|b|+a).$ Then Assumption \ref{Assumpt1} holds by  \eqref{ass1a21}, and Assumption \ref{Assumpt2} holds trivially with $c_{\varepsilon(a)}=1,G_{0}^{a}=\mathbf{1}_{[0,\infty)}$. Finally, by \eqref{a221}, Assumption \ref{Assumpt3} holds with $\tilde{\psi}$ as in \eqref{psi(aa)}.

\end{proof}
Next we come to the proof of  Theorem \ref{shockflat}. Here we need to use slow decorrelation again, as we have two maximizing paths going to the same endpoint, on the other hand, the proof is easier since the two maximizers start at points at distance $\mathcal{O}(T)$ from each other (see \eqref{distrho}), not $\mathcal{O}(T^{2/3})$ as in Theorems  \ref{Thmpp}, \ref{airydec}.
\begin{proof}[Proof of Theorem \ref{shockflat}]
We work in the last passage picture, where Theorem \ref{shockflat} is equivalent to
\begin{equation}\label{help}
\lim_{T \to \infty}\Pb(L_{\mathcal{L}^{\varrho_{1}}\cup \mathcal{L}^{\varrho_{2}}\to E}\leq T)=F_{\GOE}(2^{2/3}(s-\xi/\varrho_{1})c_1)F_{\GOE}(2^{2/3}(s-\xi/\varrho_{2})c_2),
\end{equation}
where we defined 
\begin{align}&E=((1-\varrho_{1}-\varrho_{2}+\varrho_1 \varrho_2 )T-(s-\xi)T^{1/3},\varrho_{1}\varrho_{2}T+\xi T^{1/3}) 
\\& \mathcal{L}^{\varrho_{1}}=\{(n-\lfloor n/\varrho_{1}\rfloor ,n),n\leq 0\}
\quad \mathcal{L}^{\varrho_{2}}=\{(n-\lfloor n/\varrho_{2}\rfloor ,n),n> 0\}.
\end{align} Denote by $\pi^{\mathrm{max}}_{\mathcal{L}^{\varrho_{i}}\to E},i=1,2$ the maximizing paths from $\mathcal{L}^{\varrho_{i}}\to E.$ Consider the points
\begin{equation}\label{Spoint}
S_{\varrho_{1}}=((1-\varrho_{1})(\varrho_{1}-\varrho_{2}),\varrho_{1} (\varrho_{2}-\varrho_{1}))T \quad S_{\varrho_{2}}=((1-\varrho_{2})(\varrho_{2}-\varrho_{1}), (\varrho_{1}-\varrho_{2})\varrho_{2})T.
\end{equation} 

The lines $\overline{S_{\varrho_{i}} E}$ are the characteristic lines of $\pi^{\mathrm{max}}_{\mathcal{L}^{\varrho_{i}}\to E}.$  Note that 

\begin{equation}\label{distrho}||S_{\varrho_{1}}-S_{\varrho_{2}}||_{2}=\mathcal{O}((\varrho_{1}-\varrho_{2})T).\end{equation}
That \eqref{distrho} is of order $T$ (and not $T^{2/3}$) is the main difference to all other situations considered in this paper, and implies that $L_{\mathcal{L}^{\varrho_{1}}\to E},L_{\mathcal{L}^{\varrho_{2}}\to E}$ decouple already in the  $T \to \infty$ limit.
Define the point 
\begin{equation}\label{eee}E^{\varrho_{2}}=E-((1-\varrho_{2})^{2}T^{\nu},\varrho_{2}^{2}T^{\nu}),\nu \in (2/3,1),\end{equation} which lies  on $\overline{S_{\varrho_{2}} E}$.
We can then localize the starting point of $\pi^{\mathrm{max}}_{\mathcal{L}^{\varrho_{2}}\to E^{\varrho_{2}}}$ by (4.18) of \cite{FO17}, which gives that 
\begin{equation}\label{H2}
H_{2}=\{\pi^{\mathrm{max}}_{\mathcal{L}^{\varrho_{2}}\to E^{\varrho_{2}}}(0)\in \{P\in \mathcal{L}^{\varrho_{2}}: ||P- S_{\varrho_{2}}||_{2}\leq k T^{2/3}\}\}
\end{equation}
has probability $\Pb(H_{2})>1-e^{ck^{2}}$ for $T,a$ large enough. For our purposes it suffices to know that for any $\varepsilon\in (0,1/3)$  we have that 
\begin{equation}
\tilde{H}_{2}=\{\pi^{\mathrm{max}}_{\mathcal{L}^{\varrho_{2}}\to E^{\varrho_{2}}}(0)\in \{P\in \mathcal{L}^{\varrho_{2}}: ||P- S_{\varrho_{2}}||_{2}\leq  T^{2/3+\varepsilon}\}\}
\end{equation}
satisfies  $\Pb(\tilde{H}_{2})\to_{T \to \infty}1.$ 
Exactly as in the proof of Theorem  \ref{airydec}, on the event   $\tilde{H}_{2},$ we can bound the transversal fluctuations of $\pi^{\mathrm{max}}_{\mathcal{L}^{\varrho_{2}}\to E^{\varrho{2}}}$ by the transversal fluctuations  of $\pi^{\mathrm{max}}_{P^{\prime}\to  E^{\varrho_{2}}}$ where $P^{\prime}$ is the bottom right end point of the line segment $\{P\in \mathcal{L}^{\varrho_{2}}: ||P- S_{\varrho_{2}}||_{2}\leq  T^{2/3+\varepsilon}\}:$ By Theorem \ref{trans1}, the transversal fluctuations of $\pi^{\mathrm{max}}_{P^{\prime}\to  E^{\varrho_{2}}}$ around $\overline{P^{\prime} E^{\varrho_{2}}}$ are $o(T^{2/3+\varepsilon})$ for any $\varepsilon >0$.

We can localize the starting point  $\pi^{\mathrm{max}}_{\mathcal{L}^{\varrho_{1}}\to E}(0)$ in exactly the same way, as well as the transversal fluctuations of $\pi^{\mathrm{max}}_{\mathcal{L}^{\varrho_{1}}\to E}$.
If we choose $1>\nu >2/3+\varepsilon,$  then  we have  shown there are disjoint deterministic  sets $D_{\varrho_{1}},D_{\varrho_{2}},$ such that 
\begin{equation}\label{jjj}
\lim_{T\to \infty}\Pb(\pi^{\mathrm{max}}_{\mathcal{L}^{\varrho_{2}}\to E^{\varrho_{2}}}\subset D_{\varrho_{2}}, \pi^{\mathrm{max}}_{\mathcal{L}^{\varrho_{1}}\to E}\subset D_{\varrho_{1}})=1.
\end{equation}
This implies 
\begin{align}
\lim_{T \to \infty}\Pb(L_{\mathcal{L}^{\varrho_{1}}\cup \mathcal{L}^{\varrho_{1}}\to E}\leq T)
&=\lim_{T \to \infty}\Pb(L_{\mathcal{L}^{\varrho_{1}}\to E}\leq T,L_{\mathcal{L}^{\varrho_{2}}\to E}\leq T)
\\&=\lim_{T \to \infty}\Pb(L_{\mathcal{L}^{\varrho_{1}}\to E}\leq T)\Pb(L_{\mathcal{L}^{\varrho_{2}}\to E}\leq T)
\end{align}
(the first identity holds by definition, the second by \eqref{jjj} and slow decorrelation, which implies that $T^{-1/3}(L_{\mathcal{L}^{\varrho_{2}}\to E^{\varrho_{2}}}  +\mu_{0}T^{\nu}- L_{\mathcal{L}^{\varrho_{2}}\to E})$ converges to $0$ in probability for the correct $\mu_{0}$). So what remains to compute is 
\begin{align}
\lim_{T \to \infty}\Pb(L_{\mathcal{L}^{\varrho_{1}}\to E}\leq T)&\Pb(L_{\mathcal{L}^{\varrho_{2}}\to E}\leq T)\\&=F_{\GOE}(2^{2/3}(s-\xi/\varrho_{1})c_1)F_{\GOE}(2^{2/3}(s-\xi/\varrho_{2})c_2),
\end{align}
which has, (with the $\mathcal{O}(T^{1/3})$ terms taken from $E$ into the scaling),  been proven in Lemma 2.4 of \cite{FGN17}.
\end{proof}

Now let us come to the statement   \eqref{you}. We follow the proof of Theorem \ref{shockflat}. We choose 
the point  \begin{equation}
\tilde{E}=((1-\varrho_{1}(a)-\varrho_{2}+\varrho_{1}(a) \varrho_{2})T-\xi T^{2/3}(\frac{1}{a\varrho_{2}}-\frac{1}{a})-s T^{1/3}/c_{2}, \varrho_{1}(a) \varrho_{2}T+\xi T^{2/3}/a).
\end{equation}

The first difference arises in  \eqref{help}: It is  unclear what 
\begin{equation}\label{unclear}
\lim_{T \to \infty}\Pb(L_{\mathcal{L}^{\varrho_{1}(a)}\cup \mathcal{L}^{\varrho_{2}}\to \tilde{E}}\leq T)
\end{equation}
is, and a priori it is even unclear  whether \eqref{unclear} exists as a limit in distribution. Knowing this is however not needed to study the $a\to \infty$ asymptotics, 
we may  simply consider an arbitrary subsequential limit $T_{k}\to \infty$  and show  that we have 
\begin{equation}\begin{aligned}\label{fix}
 \lim_{a \to + \infty} \lim_{T_{k} \to \infty}&\Pb(L_{\mathcal{L}^{\varrho_{1}(a)}\cup \mathcal{L}^{\varrho_{2}}\to \tilde{E}}\leq T_{k})
 \\&= F_{\GOE}(2^{2/3}s )F_{\GOE}(2^{2/3}s-2^{2/3}\xi (1-\varrho_{2})^{-2/3}\varrho_{2}^{-5/3}).
\end{aligned}
\end{equation}
Now to show \eqref{fix}, we again need to choose the point $E^{\varrho_{2}}$ from \eqref{eee}, but as in the proof of Theorem \ref{Thmpp}, $E^{\varrho_{2}}$ needs to have distance $\mathcal{O}(T)$ from $\tilde{E}$, we may choose 
$E^{\varrho_{2}}=\tilde{E}-((1-\varrho_{2})^{2}Ta^{-1/2},\varrho_{2}^{2}Ta^{-1/2}).$  With this choice, we have an extended slow decorrelation result  in the double limit $\lim_{a \to +\infty}\lim_{T \to \infty}$.

Next note that the maximizers  $\pi^{\mathrm{max}}_{\mathcal{L}^{\varrho_{2}}\to E^{\varrho_{2}}},\pi^{\mathrm{max}}_{\mathcal{L}^{\varrho_{1}(a)}\to \tilde{E}}$ 
start in $\mathcal{O}(T^{2/3})$ neighborhoods of the  points  $S_{\varrho_{1}(a)},S_{\varrho_{2}}$ (defined in \eqref{Spoint}), which themselves have distance $CaT^{2/3}+o(T^{2/3}),C>0,$  from each other (see \eqref{distrho}).
Doing a refined localization of the starting points as in \eqref{H2} with e.g. $k=a^{1/10}$, and bounding the transversal fluctuations, we see that  there are disjoint deterministic sets $\tilde{D}_{\varrho_{1}(a)},\tilde{D}_{\varrho_{2}}$ such that 
\begin{equation}
\lim_{a \to + \infty}
\liminf_{T\to \infty}\Pb(\pi^{\mathrm{max}}_{\mathcal{L}^{\varrho_{2}}\to E^{\varrho_{2}}}\subset \tilde{D}_{\varrho_{2}}, \pi^{\mathrm{max}}_{\mathcal{L}^{\varrho_{1}(a)}\to \tilde{E}}\subset \tilde{D}_{\varrho_{1}(a)})=1, 
\end{equation}
which implies 
\begin{equation}
\begin{aligned}
 \lim_{a \to + \infty} \lim_{T_{k} \to \infty}\Pb(L_{\mathcal{L}^{\varrho_{1}(a)}\cup \mathcal{L}^{\varrho_{2}}\to \tilde{E}}\leq T_{k})=\lim_{a \to + \infty} \lim_{T_k \to \infty} &\Pb(L_{\mathcal{L}^{\varrho_{1}(a)} \to \tilde{E}}\leq T_k)
\\&\times\Pb(L_{ \mathcal{L}^{\varrho_{2}}\to \tilde{E}}\leq T_k).
\end{aligned}
\end{equation}
Finally, we need to show that the individual $ \Pb(L_{\mathcal{L}^{\varrho_{1}(a)}\to \tilde{E}}\leq T), \Pb(L_{\mathcal{L}^{\varrho_{2}}\to \tilde{E}}\leq T)$ converge to $F_{\GOE}$  in the double limit $\lim_{a \to + \infty}\lim_{T \to \infty}$ . The localization of the starting point shows that when considering the full lines $ \mathcal{L}^{\varrho_{1}(a),\infty}=\{(n-\lfloor n/\varrho_{1}(a)\rfloor ,n),n \in \Z\}
\quad \mathcal{L}^{\varrho_{2},\infty }=\{(n-\lfloor n/\varrho_{2}\rfloor ,n),n\in \Z \}$ we have that 
\begin{equation}
\lim_{a \to + \infty} \limsup_{T \to\infty}\Pb(\{L_{\mathcal{L}^{\varrho_{1}(a)}\to \tilde{E}}\neq L_{\mathcal{L}^{\varrho_{1}(a),\infty}\to \tilde{E}} \}\cup\{L_{\mathcal{L}^{\varrho_{2}}\to \tilde{E}}\neq L_{\mathcal{L}^{\varrho_{2},\infty}\to \tilde{E}}\})=0.
\end{equation}
Now the convergence, for $a$ fixed, of the properly rescaled $L_{\mathcal{L}^{\varrho_{1}(a),\infty}\to \tilde{E}},L_{\mathcal{L}^{\varrho_{2},\infty}\to \tilde{E}}$  to $F_{\GOE}$ should  be  deduced  from Theorem 2.1 of \cite{FO17}.  This would  then prove \eqref{fix}.

\bibliographystyle{patplain}

\bibliography{Biblio}

\end{document}